%% file: szego.tex
\documentclass{amsart}[12pt]
\usepackage{amsmath, amsthm, amssymb} 	
\usepackage{enumerate}
\usepackage{verbatim}
\usepackage{url}
\usepackage{graphicx}
\usepackage{geometry}
\usepackage{framed}
\usepackage{tikz}
\usepackage{tikz-cd}							
	\usetikzlibrary{decorations.pathmorphing}	

\usepackage{stmaryrd}
\usepackage{mathrsfs}

\usepackage{rtd-szego}

\usepackage{hyperref}

\numberwithin{equation}{section}	

\begin{document}
\setlength{\parindent}{0cm}
\setlength{\parskip}{0.3cm}

\title{Lax formulation for harmonic maps to a moduli of bundles}

\author{Richard Derryberry}



\date{\today}

\begin{abstract}
I define an algebraic metric and closed 3-form on a subspace $\mc{M}$ of the moduli of $G$-bundles on a complex projective curve $C$, and show that the resulting two-dimensional $\sigma$-model with target $\mc{M}$ has a zero-curvature formulation.
\end{abstract}

\maketitle



\input{section-intro}
\input{section-problem-setup}

\input{section-metric}
\input{section-3form-defn}
\input{section-action}

\input{section-cxn-target}

\input{section-cxn-induced}

\input{section-main-theorem}

\input{section-CP1-eg}




\medskip


\bibliography{szego-bib}
\bibliographystyle{alpha}

\end{document}

%% file: section-intro.tex

\section{Introduction}\label{s:intro}

Harmonic maps are mappings between pseudo-Riemannian manifolds that satisfy a certain generalisation of Laplace's equation. The exact partial differential equation can be derived as the critical points of an action functional. Usually this functional is taken to be the Dirichlet energy, however I will take the slightly broader view that the functional merely needs to be the action functional of a classical $\sigma$-model:\footnote{It is not important for this paper to know what this shade of physical inspiration means, but for the curious: A $\sigma$-model is a (classical or quantum) field theory whose fields are maps $\sigma: \Sigma \to \mc{M}$, where $\Sigma$ is a `spacetime' manifold, and $\mc{M}$ is the `target' manifold.} e.g.\ the functional considered in this paper can be found at \eqref{eq:action}.

Harmonic maps from a surface to a Lie group are famously integrable, in the sense that they admit a Lax, or zero-curvature, formulation \cite{Poh76}. There is a rich literature on integrable harmonic map equations, ranging familiar coset models (particularly Riemannian symmetric spaces) \cite{Byk16,Zar19} to affine Gaudin models \cite{DLMV19a,DLMV19b} and integrable $\mc{E}$-models \cite{LV21}. However, in all of these examples the target space is (topologically!) a group manifold or coset space, and the spectral curve is genus zero. The core novelty of this paper is the expansion of possible target space topologies to certain moduli spaces of bundles, and the existence of integrable $\sigma$-models with spectral curve of arbitrary genus.

In this paper I will introduce an infinite number of new examples, approximately parametrized by the choice of a Riemann surface equipped with a holomorphic 1-form with simple zeroes and double poles. The target space is an open subspace of the moduli of bundles on the Riemann surface, which I will show may be endowed with an algebraic metric (so that pseudo-Riemannian manifolds may be obtained by taking a real slice of the target space). Algebraic metrics are significantly rarer than their pseudo-Riemannian counterparts, and the existence of one on the moduli of bundles is certainly unexpected. Another novelty of the construction I will present is that the spectral curves of the integrable systems may be of arbitrary genus, and not just of genus 0 (as is the case for the previously known examples in the literature). The construction in this paper is based on the physical engineering of a class of new integrable $\sigma$-models from 4d Chern-Simons theory by Costello and Yamazaki \cite{Costello:2019tri}, which I will briefly review in Section \ref{s:setup}.


\subsection{Integrable systems}\label{ss:int_sys}

Having introduced the concept without giving a precise definition, it seems prudent to ask to the question: what is an integrable system? A first approximation to a definition might be as follows: an integrable system is a \emph{system} of differential equations that can be \emph{integrated}, or in plain language, solved. Of course this naive definition is deficient: it captures too many systems, while failing to describe the nature or method of solution.

As a warm up to integrable systems in classical field theory, let's consider integrable systems in classical mechanics. The dynamics of a system in classical mechanics is described by a path in \emph{phase space}, a $2n$-dimensional symplectic manifold that describes all possible states of the physical system. The symplectic form gives rise to a Poisson bracket on functions $\{-,-\}$, and $\{f,g\} = 0$ if and only if the function $f$ is constant along the level sets of $g$ (and vice-versa). The physical system is described by a particular function call the Hamiltonian $H$, and a function that Poisson commutes with $H$ is called a \emph{conserved quantity} or \emph{integral of motion}.

Integrability in this context is the following statement: that one can find $n$ algebraically independent Poisson commuting conserved quantities. This allows one to solve the system of differential equations determined by $H$ as follows:
\begin{itemize}
    \item By the Arnold-Liouville theorem, there exists a canonical collection of coordinates on the phase space called \emph{action-angle} coordinates \cite{Arn89}. Roughly, the action coordinates parametrize the base of possible values of our $n$ conserved quantities, while the angle coordinates\footnote{If our fibres are compact, they will be real tori.} parameterize the fibres over each point in the base.
    \item In order to have a well-posed problem, we need to supplement our system of differential equations with $2n$ pieces of initial data. Start by choosing $n$ of these to be value in our base of conserved quantities. Since they are conserved, the path that describes the solution to our system of differential equations must stay in the fibre over this point: i.e.\ we have fixed the action coordinates, and the dynamics is now entirely concentrated in the angle coordinates.
    \item The path that describes the solution will vary linearly in the angle coordinates. Et voil\`a! We have determined the path in phase space that solves our system of differential equations with given initial conditions!
    \end{itemize}

For this paper, I'm interested not in classical mechanical systems, but in classical \emph{field} theories. Here the situation is trickier: the phase space of such a theory is infinite dimensional, and so one must try to cook up \emph{infinitely many} Poisson commuting functions on phase space and attempt to use these to solve the equations of motion, for instance via the inverse scattering method \cite{FT86}.

The key concept underlying the integrability of these infinite dimensional dynamical systems is that of the \emph{Lax pair}, or \emph{zero-curvature}, formulation of the system. The basic idea here is to determine a recipe that for each classical field $\sigma$ produces a connection $D(\sigma)$ on spacetime, such that $\sigma$ satisfies the classical equations of motion if and only if the connection $D(\sigma)$ is flat. Integrals of motion may then be obtained by applying invariant polynomials to the monodromies of this connection.

As described above, we would still only obtain finitely many integrals of the motion for our infinite dimensional system. To remedy this, one usually produces for each $\sigma$ not a single connection but a \emph{family} of connections $D(\sigma)_z$ parametrized by $z\in C$, the \emph{spectral parameter} on an algebraic curve called the \emph{spectral curve}. The flatness condition becomes the requirement that $D(\sigma)_z$ is simultaneously flat for all values of $z\in C$. Applying invariant polynomials to the monodromies of this family of connection then produces infinitely many integrals of the motion; for instance, if $C=\mb{C}^\times$ one can obtain a countable family of commuting conserved quantities as the coefficients of a power series expansion of\footnote{Or a function derived from this one.} $\Tr\Hol D(\sigma)_z$ around $z=0$ or $z=\infty$ \cite{FT86}.


\subsection{An explicit example}\label{ss:explicit}

For concreteness, let's suppose that our spacetime is $\Sigma = S^1\times\mb{R}$ (or in the algebraic setup of this paper, $\mb{C}^\times \times \widehat{\mb{D}}$). Let $\mc{M}$ be any space for which the harmonic map equations are integrable, and let $C$ be the spectral curve; for example, $\mc{M} = G$ and $C = \mb{C}^\times$.  Let $\Harm(\Sigma)$ denote the moduli of solutions to the harmonic map equations.

Then for each field $\sigma\in\Harm(\Sigma)$ there is a family of flat connections $D(\sigma)_z$, $z\in C$, on spacetime. We can take its holonomy around $\mb{C}^\times$ (the exact loop does not matter due to flatness), and apply an invariant function $f\in \mb{C}[G]^G$ to get a number
    \[   \mc{F}_{f,z}(\sigma) = f(\Hol_{\mb{C}^\times}D(\sigma)_z) \in \mb{C}.   \]
As $\sigma$ varies this gives a function $\mc{F}_{f,z}: \Harm(\Sigma) \to \mb{C}$ for each invariant function $f$ and point $z$ in the spectral curve. These functions are all conserved by flatness of the connections $D(\sigma)$, and Poisson commute with each other \cite{BBT03}. So we've found infinitely many commuting conserved quantities.

It remains to ask: what is the space $\Harm(\Sigma)$? Working on the formal disc, we can expand a solution as a formal power series with coefficients in the algebraic loop group of the target space $\mc{M}$. Because the equations of motion are second order, the solution only depends on a point and a first order variation in the loop group. Hence $\Harm(\Sigma) \cong TL\mc{M}$. The metric on $\mc{M}$ induces a metric on $L\mc{M}$, so we can identify this with $T^\ast L\mc{M}$. This suggests that we have found infinitely many conserved quantities for a quantum mechanical system on the loop space of $\mc{M}$.

\begin{remark}\label{rk:caution}
An important word of caution: I say `suggests' above because it is not clear that the Poisson structure obtained from the harmonic map equations and the canonical Poisson structure on $T^\ast L\mc{M}$ agree.
\end{remark}


\subsection{Future directions}\label{ss:future}

The work in this paper represents only the beginning of what could be studied for these harmonic map equations:
\begin{itemize}
    \item The existence of a zero-curvature formulation with higher genus spectral curve $C$ suggests that there may be an action on the space of solutions by the group of maps from $C$ into $G$. For genus zero this is the entry point for the loop group action on solutions that allows one to cook up non-trivial solutions and understand the geometry of the moduli space of harmonic maps \cite{Uhl89}. The hope is that a similar action could be exploited to similar effect in the higher genus case.
    \item The 1-loop $\beta$-function of the $\sigma$-model corresponding to our harmonic map equations is a modified Ricci flow equation on the target space \cite{CFMP85}. A priori this modified Ricci flow equation is difficult to understand; however Costello has communicated to me the following conjecture regarding its behaviour:
    \begin{conj}[Costello]\label{conj:cost}
        Let $\mc{N}$ denote the moduli space of Riemann surfaces equipped with a holomorphic 1-form $\omega$ that has only simple zeroes and double poles, together with a decomposition of the zeroes of $\omega$ into two equally sized groups, $D_1$ and $D_2$.

        Then the modified Ricci flow on the space of metrics on the target is identical to the flow generated by a certain flow on $\mc{N}$, where:
        \begin{itemize}
            \item The closed periods $\oint\omega$ are held fixed.
            \item The periods $\int_p^q \omega$ are held fixed when $p,q$ are both in either $D_1$ or $D_2$.
            \item When $p\in D_1$ and $q\in D_2$ the periods $\int_p^q \omega$ satisfy $\left.\frac{d}{d\epsilon}\int_p^q\omega\right|_{\epsilon=0} = 1$, where $\epsilon$ is the parameter of the flow on $\mc{N}$.
        \end{itemize}
    \end{conj}
    Some evidence for this conjecture is presented in section \ref{ss:CP1-betafnc}.
    \item Finally, the flat connections on our target space can be interpreted via pushforward as D-modules on $\BunSt_G(C)$. We can therefore ask the standard question: what are the corresponding Langlands dual objects?
\end{itemize}


\subsection{Structure of the paper}\label{ss:outline}

The structure of the paper is as follows:

In Section \ref{s:setup} I introduce the core problem of the paper, and give a brief review of the motivating physical analysis of Costello and Yamazaki.

Sections \ref{s:metric}--\ref{s:action-variation} are concerned with the definition of the field theories introduced by Costello and Yamazaki. These theories are two-dimensional $\sigma$-models, and thus require a metric (Section \ref{s:metric}) and a closed 3-form (Section \ref{s:3form-def}) be defined on the target space. Once these have been defined, one can write down the action of the classical field theory and derive its equations of motion (Section \ref{s:action-variation}). In Section \ref{s:metric} I also discuss some properties of the metric: I calculate its first derivative in a natural collection of coordinates, and consider the signature of an induced pseudo-Riemannian metric on a real slice of the target space.

In Sections \ref{s:cxn-target}--\ref{s:cxn-induced} I consider some novel $G$-connections that were predicted by Costello and Yamazaki. In Section \ref{s:cxn-target} I show that there are two different flat algebraic $G$-connections on the target of the $\sigma$-model; in Section \ref{s:cxn-induced} I use these two connections to build a $G$-connection on spacetime for every field in the $\sigma$-model.

In Section \ref{s:main-theorem} I integrate the previous two topics to prove that the classical equations of motion of the $\sigma$-models have a zero-curvature formulation.

Finally, in Section \ref{s:CP1-eg} I consider in detail the example of $\mb{CP}^1$ with two marked points. In special cases this example recovers the motivating example of harmonic maps into Lie groups.


\subsection{Acknowledgements}\label{ss:ack}

I sincerely thank my postdoctoral mentor Kevin Costello for his support during this project, as well as for suggesting the problem; Benoit Vicedo for a productive discussion on a previous version of this paper; David Ben-Zvi, with whom I had multiple extremely useful correspondences; and Tom Mainiero and Sebastian Schulz for comments on a draft version of this paper. Finally, I would like to thank Zahavi Derryberry for annotating my notes, and Felix Derryberry for providing typing assistance.

Research at Perimeter Institute is supported in part by the Government of Canada through the Department of Innovation, Science and Economic Development and by the Province of Ontario through the Ministry of Colleges and Universities.


%

%

%% file: section-problem-setup.tex

\section{Setup and background of the problem}\label{s:setup}

Our starting data is a proper curve $C$ over $\mb{C}$ of genus $g$, equipped with a meromorphic 1-form $\omega$ which has only simple zeros and double poles. Let $Q$ be half the divisor of poles of $\omega$, and divide the divisor of zeroes into two equal degree divisors $P_1$ and $P_2$. We then form the two divisors $D_{1,2} = P_{1,2} - Q$ on $C$ of degree $g-1$ which satisfy $\mc{O}(D_1+D_2) = K_C$.

Define $\BunSt_G(C|D_i)\subset\BunSt_G(C,Q)$ to be the open substack of the moduli of $G$-bundles on $C$ trivialised at $Q$ satisfying the cohomology vanishing condition $H^\bullet(C;\mf{g}_P(D_i))=0$. Note that by our assumptions and Serre duality,
	\[	\BunSt_G(C|D_1)=\BunSt_G(C|D_2) =: \mc{M}.	\]

Finally, let $\mc{P}$ denote the universal $G$-bundle on $C\times\BunSt_G(C,Q)$, and let $C_i := C\setminus D_i$, $C_0 := C_1\cap C_2$.


\subsection{Review of the physical problem}\label{ss:physics-review}

My claim, based on the physical analysis of Costello and Yamazaki in \cite{Costello:2019tri}, is that from this initial data we can construct two a priori unrelated doohickeys:
\begin{itemize}
	\item A two-dimensional classical Lagrangian field theory, specifically a type of $\sigma$-model with target $\mc{M}$. In particular, our starting data allows us to construct a metric and closed 3-form on $\mc{M}$.
	\item For each field $\sigma$ in the $\sigma$-model, a connection $D(\sigma)$ on spacetime.
\end{itemize}
Then the physical analysis in \cite{Costello:2019tri} concludes, and the main theorem of this paper proves, that a field $\sigma$ is a solution to the classical equations of motion for the $\sigma$-model if and only if $D(\sigma)$ is a flat connection.

My approach to the problem will be different to that of Costello and Yamazaki. For the sake of history and motivation, however, let us briefly review their argument:

Costello and Yamazaki begin with the 4d Chern-Simons theory introduced in \cite{C13}. This is the four-dimensional theory with action
\begin{align}\label{eq:CSaction}
	S_{CS}[A] = \frac{1}{2\pi\hbar}\int_{\mb{R}^2\times C} \omega \wed CS(A),
\end{align}
where $A$ is a gauge field and $CS(A) = \tr(A\wed dA + \frac{2}{3}A\wed A \wed A)$. Equipping $\mb{R}^2$ with a complex coordinate $w, \bar{w}$ we may write
	\[	A = A_w dw + A_{\bar w}d\bar w + A_{\bar z}	\]
i.e.\ we assume that $A$ has no $(1,0)$-component in the $C$-direction. The equations of motion for the theory are precisely the flatness equations $F(A) = 0$.

The upshot of \cite{Costello:2019tri} is that coupling the action \eqref{eq:CSaction} to various surface defects and compactifying on the curve $C$ leads to a wide variety of two-dimensional classical field theories which are automatically integrable due to features inherited from the 4d theory. `Integrable' in this situation means `has a Lax/zero curvature formulation'; hence in this framework the main theorem of this paper comes for free once one knows that the 2d theory of interest may be engineered from 4d Chern-Simons theory.

\begin{remark}\label{rk:previous-CP1-work}
In the situation where $C = \mb{CP}^1$ is the Riemann sphere, this procedure recovers the integrable $\sigma$-models originally constructed by different means in \cite{DLMV19a,DLMV19b}.
\end{remark}

To obtain the theory of interest in the Costello-Yamazaki framework, one begins by imposing certain boundary conditions on the gauge fields and gauge transformations in order to account for the poles and zeroes of $\omega$:
\begin{itemize}
	\item At each simple zero of $\omega$, either $A_w$ or $A_{\bar w}$ has a simple pole.
	\item At each pole of $\omega$, $A$ vanishes.
	\item We only allow gauge transformations that vanish at the poles of $\omega$.
	\end{itemize}

Assume for simplicity that we are studying connections on the trivial principal bundle. Then we can consider the component $A_{\bar z}(w,\bar{w})$ as an $\mb{R}^2$-family of holomorphic structures on the bundle via the family of operators $\delbar + A_{\bar z}(w,\bar w)$; this shows us that one of the fields in our compactified theory will be a map $\mb{R}^2\to \BunSt_G(C,Q)$. For a large class of such maps---namely, those that land in $\mc{M}$---one can further solve uniquely for the fields $A_w$ and $A_{\bar w}$. Thus if we impose this restriction on the target space, we see that our 2d compactification is a $\sigma$-model; the Lagrangian of this theory can be determined by inserting the expressions for $A_w, A_{\bar w}$ in terms of $A_{\bar z}$ back into the original Lagrangian.

In this paper I will show that the formulae for the metric, 3-form, and connections described by Costello and Yamazaki are well-defined, and by direct calculation will verify that the conclusion about the zero-curvature formulation of the classical equations of motion holds.


%

%

%% file: section-metric.tex

\section{The metric and its properties}\label{s:metric}


\subsection{Definition of the metric}\label{ss:metric-def}

First, let us define the metric. I will begin by giving the ``cleanest'' definition of the metric as a manifestly algebraic pairing on cohomology groups, before mentioning two cochain models that can be used for computations.

Consider the short exact sequence of sheaves on $C$,
	\[	0 \to \mf{g}_P(-Q) \to \mf{g}_P(D_1) \to \mf{g}_P\tens\mc{O}_{P_1}(P_1)\to 0	\]
which, by our cohomology vanishing assumption, implies the isomorphism
	\[	H^1(C;\mf{g}_P(-Q))\cong
		H^0(C;\mf{g}_P\tens\mc{O}_{P_1}(P_1))\cong
		\bigoplus_{x\in P_1} H^0(C;\mf{g}_P\tens\mc{O}_x(x))	\]
I will describe a pairing that is block diagonal with respect to this direct sum decomposition as follows. Let $\la-,-\ra$ denote a nondegenerate invariant pairing on $\mf{g}$. Then if we are working over the field $\mb{F}$, we define an $\mb{F}$-valued pairing on each block as the following composition:
\[\begin{tikzcd}
	H^0(C;\mf{g}_P\tens\mc{O}_x(x))^{\tens 2}\ar[r,"\la\phantom{ab}\ra"]
	& H^0(C;\mc{O}_x(x))^{\tens 2}\ar[r]
	& H^0(C;\mc{O}_x(2x)) \ar[r,"\cdot\omega"]
	& H^0(C;K_x(x))\ar[r,"\Res_x"]
	& \mb{F}
\end{tikzcd}\]
Note that this is just the value at $P\in\mc{M}$ of an algebraic pairing that can be defined on certain sheaves on $C\times\mc{M}$ pushed forward along the projection to $\mc{M}$. It is symmetric because $\la-,-\ra$ is symmetric and multiplication is symmetric in a commutative ring, and nondegeneracy follows from nondegeneracy of the pairing $\la-,-\ra$. Hence this pairing is an algebraic metric on $\mc{M}$.


\subsubsection{Dolbeault model for the metric}\label{sss:dolbeault-metric}

I will now describe the Dolbeault model for the metric, which will be used in later calculations. Consider the diagram
\[\begin{tikzcd}
	0\ar[r]
		& \Omega_C^{0,0}(\mf{g}_P(-Q))\ar[r,"\rho_i"]\ar[d,"\delbar"]
			& \Omega_C^{0,0}(\mf{g}_P(D_i))\ar[d,"\delbar"] \\
	0\ar[r]
		& \Omega_C^{0,1}(\mf{g}_P(-Q))\ar[r,"\rho_i"]
			& \Omega_C^{0,1}(\mf{g}_P(D_i))\ar[u,bend left,"\delbar_i\inv"]
\end{tikzcd}\]
Let $\la-,-\ra$ denote the same nondegenerate invariant pairing as before. We define the metric at the point $P\in\mc{M}$ by
\begin{align}\label{eq:g}
	A_1,A_2
		\mapsto
		g_P(A_1,A_2):=
		\int_C \omega\wed\la
			\delbar_1\inv\rho_1(A_1)\tens A_2
			+
			\delbar_1\inv\rho_1(A_2)\tens A_1
		\ra,
		\quad
		A_1,A_2\in\Omega_C^{0,1}(\mf{g}_P(-Q)).
\end{align}

\begin{remark}\label{rk:omit-rho}
For the rest of the paper I will leave implicit the embeddings $\rho_i$.
\end{remark}

\begin{prop}\label{p:g-metric}
$g$ defines a metric on $\mc{M}$.
\end{prop}

In order to prove this proposition, let me recall that $\delbar_{i,P}\inv$ ($i=1,2$) can be expressed in terms of an integral kernel, the \emph{Szeg\"o kernel $\mc{S}_P$} \cite[Def.\ 15.2]{Costello:2019tri}.

\begin{defn}[Szeg\"o kernel]
The \emph{Szeg\"o kernel for $P\in\mc{M}$} is the unique element $\mc{S}_P\in H^0(C\times C;\mf{g}_P\tens\mf{g}_P\tens\mc{O}(D_1\times C + C\times D_2 + \Delta_C))$ such that the residue of $\mc{S}_P$ along the diagonal is the quadratic Casimir for $\mf{g}$.
\end{defn}

\begin{proof}
In order to prove that $g$ defines a metric on $\mc{M}$ we need to show the following:
\begin{enumerate}
	\item $g_P$ is smooth in $P$.
	\item $g_P$ is gauge invariant, i.e.\ it vanishes if one of the $A_i$ is a coboundary (hence the expression descends to a pairing on cohomology).
	\item The resulting pairing on cohomology is nondegenerate.
\end{enumerate}
First, the Szeg\"o kernel varies algebraically in $P$ \cite[Prop.\ 5.6]{BZB03}, and so the expression $\delbar_{i,P}\inv A$ varies smoothly with $P\in\mc{M}$. Hence we have that the expression for the metric varies smoothly on $\mc{M}$.

For gauge invariance, observe that for $\phi\in\Omega^{0,0}_C(\mf{g}_P(-Q))$ and $A\in\Omega^{0,1}_C(\mf{g}_P(-Q))$ we have that $\omega\la\delbar_1\inv A,\phi\ra$ is a smooth function (the poles of $\delbar_1\inv A$ and $\phi$ cancel with the zeroes of $\omega$, and vice versa). Hence,
\begin{align*}
	g_P(A,\delbar\phi)
		&= \int_C\omega\wed\la\delbar_1\inv A,\delbar\phi\ra
			+\int_C\omega\wed\la A,\phi\ra
		= \int_C\delbar\left(
			\omega\la\delbar_1\inv A,\phi\ra
		\right)
			= \int_{\pd C=\emptyset}\omega\la\delbar_1\inv A,\phi\ra
			=0.
\end{align*}

Finally, for nondegeneracy we will compute the metric in a convenient basis for the tangent space. Set $P_1 = p_1 + \cdots p_n$, with all $p_i$ distinct. Let $\mb{D}_i$ be a coordinate disc around $p_i$ with coordinate $z_i$ chosen such that $\omega|_{\mb{D}_i} = z_i dz_i$, and let $\{t_a\}$ be a basis of $\mf{g}$. Let $\delta_{|z_i|=\epsilon}$ be the distributional $(0,1)$-form defined by
	\[	\int g(z,\bar{z})dz\wed\delta_{|z|=\epsilon}
			= \oint_{|z|=\epsilon}g(z,\bar{z})dz	\]
(or a radially symmetric smooth mollification of said form). Trivialise $\mf{g}_P$ on the discs $\mb{D}_i$, and define
\begin{align}\label{eq:concentrated-basis}
	A_{ia} := \left\{
		\begin{array}{lr}
		\frac{t_a}{z_i}\delta_{|z_i|=\epsilon} & \text{on }\mb{D}_i, \\
		0 & \text{on } C\setminus\mb{D}_i
		\end{array}\right.
\end{align}
Note that $\delbar_1\inv A_{ia} = \frac{t_a}{z_i}\delta_{|z_i|\leq\epsilon}$ (or, again, an appropriately mollified smooth approximation of the step function) so that we have
\begin{align*}
	g_P(A_{ia},A_{jb})
		&= \int_C \omega\wed \la \frac{t_a}{z_i}\delta_{|z_i|\leq\epsilon},\frac{t_b}{z_j}\delta_{|z_j|=\epsilon} \ra
			+
			\int_C \omega\wed \la \frac{t_b}{z_j}\delta_{|z_j|\leq\epsilon},\frac{t_a}{z_i}\delta_{|z_i|=\epsilon} \ra \\
		&= \delta_{ij}\oint_{|z_i|=\epsilon} \frac{dz_i}{z_i}\la t_a,t_b\ra
			= 2\pi i\delta_{ij}\kappa_{ab}
\end{align*}
where $\kappa_{ab}=\la t_a,t_b\ra$. Thus, nondegeneracy of the metric follows from nondegeneracy of the pairing $\la-,-\ra$.
\end{proof}


\subsubsection{\v{C}ech model for the metric}\label{sss:metric-Cech}

Let us now also consider the following algebraic \v{C}ech model of the metric. Consider the \v{C}ech complex calculating the derived pushforward over an affine open $U\subset \mc{M}$, with respect to the open cover $\left\{C\setminus D_1, \coprod \mb{D}_i\right\}$ of $C$:
\[\begin{tikzcd}
	0\ar[r]
		& \check{C}^0(\ad(\mc{P})(-Q))\ar[r,"\rho_i"]\ar[d,"\delta"]
			& \check{C}^0(\ad(\mc{P})(D_i))\ar[d,"\delta"]\ar[r]
				& \bigoplus_{x\in P_1}\ad(\mc{P}|_{x\times U})\tens T_x C \ar[d]\ar[r]
					& 0 \\
	0\ar[r]
		& \check{C}^1(\ad(\mc{P})(-Q))\ar[r,"\rho_i"]
			& \check{C}^1(\ad(\mc{P})(D_i))\ar[u,bend left,"\delta_i\inv"]\ar[r]
				& 0
\end{tikzcd}\]
On this open affine, define the metric by the formula
\[	g(A_1,A_2)
			= \sum_{x\in D_1} \Res_x \left(\omega\la\delta_1\inv\rho_1(A_1),\delta_1\inv\rho_2(A_2)\ra\right).	\]
We would like to show that:
\begin{enumerate}
	\item This expression makes sense.
	\item This expression vanishes if one of the $A_i$ is a coboundary (hence the expression descends to a pairing on cohomology).
	\item The resulting pairing on cohomology is nondegenerate.
\end{enumerate}

First, the expression $\delta_1\inv(A_1)\tens\delta_1\inv(A_2)$ is an element of $\bigoplus_{x\in P_1}\ad(\mc{P}|_{x\times U})^{\tens2}\tens T_x C^{\tens2}$. Fix a point $x\in C$. Then in taking global sections of the exact sequence
	\[	0\to K_C(-2x) \to K_C(-x)\to \mc{O}_x\tens T^\ast_x C^{\tens 2} \to 0	\]
one sees that $\omega\in \Gamma(\mc{O}_C(-D_1-D_2))\subset \Gamma(\mc{O}_C(-D_1))$ determines an element $\omega_{P_1}\in \bigoplus_{x\in P_1} T^\ast_x C^{\tens 2}$. Since $\omega$ has strictly first order zeroes at all the points of $P_1$, $\omega_{P_1}$ is nonvanishing. Thus it trivialises $T_xC^{\tens 2}$, and can be contracted with $\delta_1\inv(A_1)\tens\delta_1\inv(A_2)$ to obtain an element of $\bigoplus_{x\in P_1}\ad(\mc{P}|_{x\times U})^{\tens2}$. $g(A_1,A_2)$ is then given by applying $\la-,-\ra$ and summing over the points of $P_1$ to obtain a function on $U$.

This description of the metric makes clear that once again nondegeneracy on cohomology amounts to the nondegeneracy of the invariant pairing $\la-,-\ra$. It therefore remains to show that the expression vanishes if one of the $A_i$ is coboundary. But by exactness of the top row in the above diagram, if $A_i$ is coboundary in $\check{C}^1(\ad(\mc{P})(-Q))$ it projects to zero in $\bigoplus_{x\in P_1}\ad(\mc{P}|_{x\times U})\tens T_x C$.


\subsection{Properties of the metric}\label{ss:metric-prop}

Let us now take a moment to explore some of the properties of this metric. In particular, I am interested in obtaining an expression for the first derivative of the metric in local coordinates, and an understanding of the (pseudo-)Riemannian metrics induced on real slices of $\mc{M}$.


\subsubsection{First derivative of the metric}\label{sss:metric-1deriv}

Let $\{A_{ia}\}$ be a collection of $\mf{g}_P(-Q)$-valued $(0,1)$-forms which descend to give a basis of $H^1(C;\mf{g}_P(-Q))$. (For concreteness, the reader may wish to consider the basis given by \eqref{eq:concentrated-basis}.) Then for small enough $\vec\lambda = (\lambda^{ia})$ the expression
	\[	\delbar_{P+\vec\lambda}:=\delbar_P + \lambda^{ia}A_{ia}	\]
defines a gauge-inequivalent $\delbar$-operator on the smooth bundle underlying $P$, hence $(\lambda^{ia})$ provides a system of coordinates in a neighbourhood of $P$.

Now, we can express $\delbar_{P+\vec\lambda}\inv$ as a geometric series
\begin{align*}
	\delbar_{P+\vec\lambda}\inv
			&= (\delbar_P + \lambda^{ia}[A_{ia},-])\inv
			= (1 + \lambda^{ia}\delbar_P\inv[A_{ia},-])\inv \delbar_P\inv \\
			&= \delbar_P\inv + \sum_{N\geq 1}(-1)^N \lambda^{i_1a_1}\cdots\lambda^{i_Na_N}\delbar_P\inv [A_{i_1a_1},-]\delbar_P\inv\cdots[A_{i_Na_N},-]\delbar_P\inv
\end{align*}
so that the metric components are locally of the form
\begin{align*}
	g_{ia,jb}(\vec\lambda) = g_{P+\vec\lambda}\left(\frac{\pd}{\pd\lambda^{ia}},\frac{\pd}{\pd\lambda^{jb}}\right)
		&= \int_C\omega\wed\la\delbar_{P+\vec\lambda,1}\inv A_{ia},A_{jb}\ra
			+\int_C\omega\wed\la\delbar_{P+\vec\lambda,1}\inv A_{jb},A_{ia}\ra \\
		&= g_P(A_{ia},A_{jb}) \\
		&\qquad	-\lambda^{kc}\int_C\omega\wed\la\delbar_{P,1}\inv [A_{kc},\delbar_{P,1}\inv A_{ia}],A_{jb}\ra \\
		&\qquad	-\lambda^{kc}\int_C\omega\wed\la\delbar_{P,1}\inv [A_{kc},\delbar_{P,1}\inv A_{jb}],A_{ia}\ra
			+ O(\lambda^2).
\end{align*}
So we have:
\begin{prop}\label{p:g-deriv}
The derivative of $g$ at $P$ is given in local coordinates by
	\[	\left.\frac{\pd g_{ia,jb}}{\pd\lambda^{kc}}\right|_{\vec\lambda=0}
			= -\int_C\omega\wed\la\delbar_{P,1}\inv [A_{kc},\delbar_{P,1}\inv A_{ia}],A_{jb}\ra
				-\int_C\omega\wed\la\delbar_{P,1}\inv [A_{kc},\delbar_{P,1}\inv A_{jb}],A_{ia}\ra .	\]
\end{prop}


\subsubsection{The metric on real slices}\label{sss:metric-real}

The physical system from which this project takes inspiration must of course have as its target space some real pseudo-Riemannian manifold, necessitating the choice of a real slice of our moduli space $\mc{M}$. With this in mind, let us explore the induced pseudo-Riemannian metric on slices induced by a real structure on the curve $C$ and real form of $G$.

Let $\rho_C:C\to C$ and $\rho_G:G\to G$ be antiholomorphic involutions, where $\rho_G$ is also a homomorphism defining a real form of the group. We further require that $\rho_C^\ast(\omega)=\overline{\omega}$. This gives a real structure $\rho$ on $\mc{M}$ defined by taking a bundle $P$ with transition functions $(\varphi_{ij})$ on some open cover to the bundle with transition functions $(\rho_G\circ\varphi_{ij}\circ\rho_C)$. Fixed points of this action are bundles with transition functions $(\varphi_{ij})$ such that there exists a cochain $(s_i)$ such that $s_i\varphi_{ij}s_j\inv = \rho_G\circ\varphi_{ij}\circ\rho_C$.

Take the real points $\mc{M}(\mb{R})$ with respect to this structure, and define a pseudo-Riemannian metric by taking $g_{\mb{R}}=\frac{1}{2\pi i}g$. We wish to understand the possible signatures of $g_{\mb{R}}$.

We work in a basis modelled on \eqref{eq:concentrated-basis}. Assume that $\{t_a\}$ is a basis for the real form defined by $\rho_G$. Then
	\[	\rho_C^\ast \overline{A_{ia}} = \frac{t_a}{\rho_C(\overline{z_i})}\delta_{|z_i|=\epsilon}.	\]
Now, the coordinates $z_i$ are such that $\omega=z_idz_i$, and $\rho_C^\ast(\omega)=\rho_C(z_i)d\rho_C(z_i)=\overline{z_i}d\overline{z_i}$. This restricts the form of $\rho_C$ in these coordinates to be $\rho_C(z_i)=\lambda_i\overline{z_i}$, where $\lambda_i=\pm1$. If $\lambda_i=+1$ then $A_{ia}$ is a tangent vector to $\mc{M}(\mb{R})$; if $\lambda_i = -1$ then $iA_{ia}$ is a tangent vector to $\mc{M}(\mb{R})$. Together with the fact that $\frac{1}{2\pi i}g(A_{ia},A_{jb})=\delta_{ij}\kappa_{ab}$ we have shown the following:

\begin{prop}\label{p:possible-sigs}
Let $\Sigma(\kappa)$ denote the signature of the form $\la-,-\ra$ restricted to the real form of $\mf{g}$ defined by $\rho_G$. Then the signature of $g_{\mb{R}}$ is $(\lambda_1\Sigma(\kappa),\ldots,\lambda_{g-1}\Sigma(\kappa))$.
\end{prop}

The following example demonstrates that the sequence $(\lambda_1,\ldots,\lambda_{g-1})$ can in fact take any value in $\{\pm1\}^{g-1}$.

\begin{eg}
For this example, we make use of the fact that a Riemann surface equipped with a nonzero holomorphic 1-form is equivalent to a \emph{translation surface}. An accessible and interesting introduction to translation surfaces is given by \cite{Wr15}. For the purposes of this paper, however, we simply need to know the following facts about the surface $C$ pictured in Figure \ref{fig:hexagons}:
\begin{itemize}
	\item The shaded region, considered as sitting in the complex plane, corresponds to the surface.
	\item The surface is obtained by identifying opposite edges of the rectangle and opposite edges of each hexagon.
	\item If there are $g-1$ hexagons, the resulting surface is of genus $g$.
	\item The surface $C$ is equipped with the holomorphic 1-form $\omega$ induced by the 1-form $dz$ on $\mb{C}$.
	\item The corners of the hexagons correspond to simple zeros of $\omega$.
	\item The surface $C$ inherits a real structure $\rho$ from complex conjugation on $\mb{C}$.
	\end{itemize}
\begin{figure}[htp]
	\centering
	\includegraphics[width=6in]{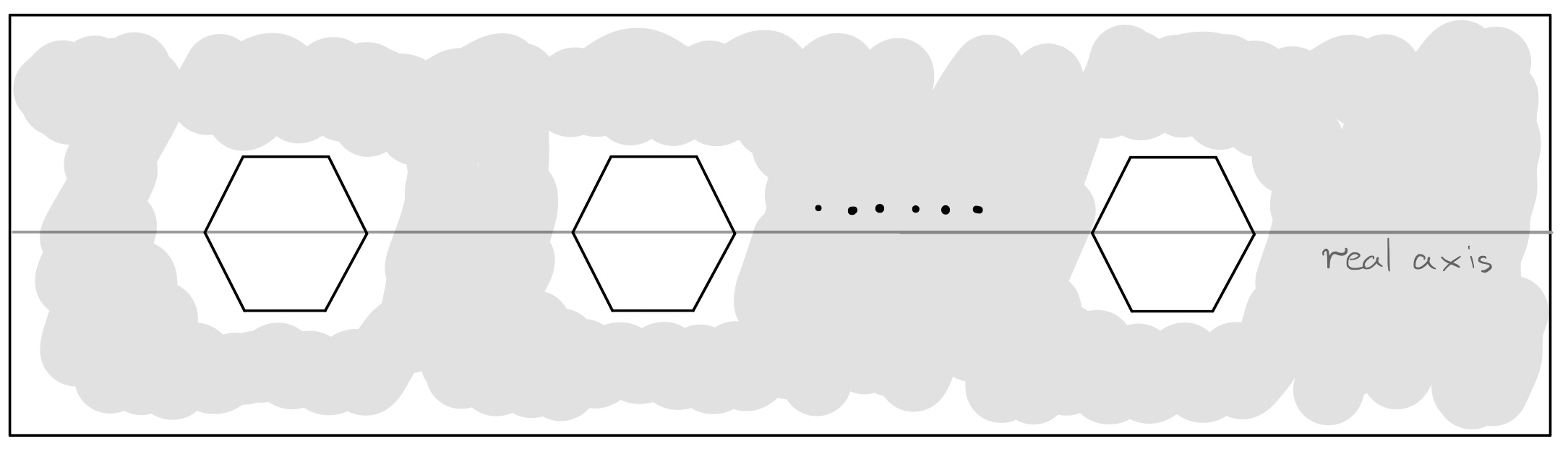}
	\caption{A genus $g$ translation surface with complex conjugation symmetry. Note there are $g-1$ hexagons.}
	\label{fig:hexagons}
	\end{figure}

In order to define our metric, we need to choose $g-1$ of the zeroes of $\omega$, assumed to lie on the real axis, to be the divisor $D_1$. Let us suppose that from each hexagon we choose either the leftmost corner or the rightmost corner. We wish to know how $\rho$ acts on coordinate $w$ centred at one of these corners.

Near one of these corners, the local coordinate satisfies $wdw = dz$, i.e.\ $w$ is proportional to a square-root of $z$. To make sense of this, we need to choose a branch cut for the square-root. Let $z=r e^{i\theta}$.

Suppose that we choose the rightmost corner, so that a local model for the corner looks like Figure \ref{fig:negbc}. Then we take our branch cut along the negative real axis, so that
	\[	w = 2\sqrt{r}e^{i\theta/2}, \quad -\pi <\theta < \pi.	\]
Then $\rho$ acts on $w$ by taking $\theta\mapsto -\theta$, so that $\rho(w) = \bar{w}$. Hence in this case, $\lambda_i = +1$.

\begin{figure}[htp]
	\centering
	\includegraphics[width=2in]{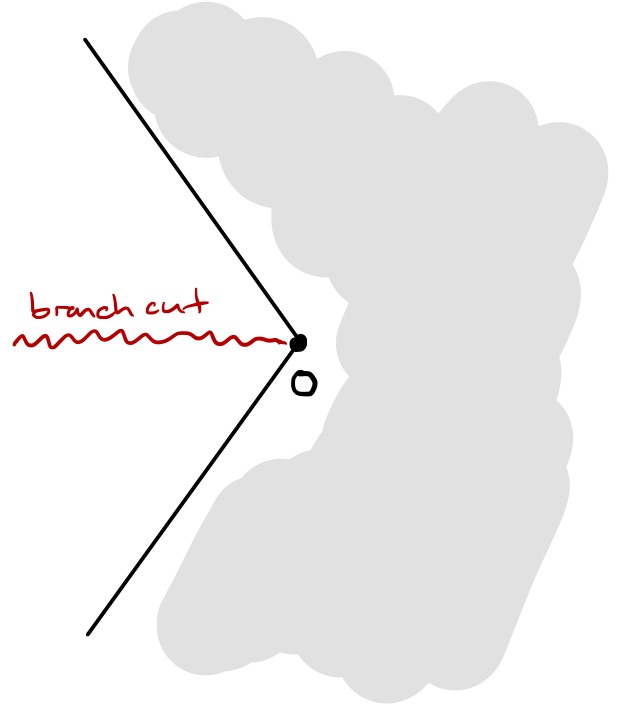}
	\caption{Taking the branch cut for the square-root along the negative real axis.}
	\label{fig:negbc}
	\end{figure}

Now suppose instead that we choose the leftmost corner, so that a local model for the corner looks like Figure \ref{fig:posbc}. Then we take our branch cut along the positive real axis, so that
	\[	w = 2\sqrt{r}e^{i\theta/2}, \quad 0<\theta<2\pi.	\]
Now $\rho$ acts on $w$ by taking $\theta\mapsto 2\pi - \theta$, so that $\rho(w) = -\bar{w}$. Hence in this case, $\lambda_i = -1$.

\begin{figure}[htp]
	\centering
	\includegraphics[width=2in]{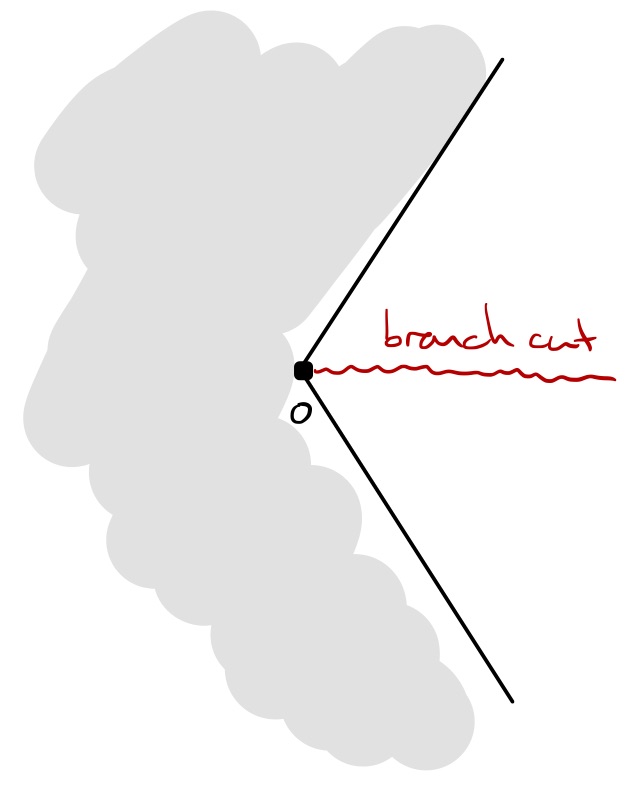}
	\caption{Taking the branch cut for the square-root along the positive real axis.}
	\label{fig:posbc}
	\end{figure}

Thus, by choosing left/rightmost points of the hexagons appropriately, any sequence of $\pm1$ of length $g-1$ may be achieved.
\end{eg}


%

%

%% file: section-3form-defn.tex

\section{Definition of the 3-form}\label{s:3form-def}

Next, let us define the 3-form that will appear in the action of our $\sigma$-model. I will define this expression on the Dolbeault complex and then show that it is gauge invariant.

Recall that there is a totally antisymmetric invariant 3-tensor $\la[-,-],-\ra \in \left(\bigwedge^3\mf{g}^\ast\right)^G$. We define the 3-form $\Omega$ on $\mc{M}$ by the formula
\begin{align}\label{eq:3form}
	A_1,A_2,A_3
		&\mapsto
			\sum_{\sigma\in S_3}(-1)^\sigma
			\int_C \omega\wed\la
				[A_{\sigma(1)},\delbar_1\inv(A_{\sigma(2)})],\delbar_2\inv(A_{\sigma(3)})
			\ra,
			& A_1,A_2,A_3\in\Omega_C^{0,1}(\mf{g}_P(-Q)).
\end{align}


\begin{prop}\label{p:3-form}
Equation \eqref{eq:3form} defines a closed 3-form on $\mc{M}$.
\end{prop}

\begin{proof}
Smoothness follows as for the metric by the expression for $\delbar_P\inv$ in terms of the Szeg\"o kernel. Next, we would like to see that \eqref{eq:3form} vanishes if one of the $A_i$ is a coboundary. Consider the expression
	\[	\omega\la
			[\delbar_1\inv(A_1),\delbar_1\inv(A_2)],\delbar_2\inv(A_3)
		\ra.	\]
Set $A_1=\delbar\phi$, $A_2=A$,$A_3=B$. Then the expression $\omega\la[\phi,\delbar_1\inv A],\delbar_2\inv B\ra$ is holomorphic, and so we have
\begin{align*}
	0 	&= \int_C\delbar\left(
				\omega\la[\phi,\delbar_1\inv A],\delbar_2\inv B\ra
			\right) \\
		&= -\int_C\omega\wed\la[\delbar\phi,\delbar_1\inv A],\delbar_2\inv B\ra
			-\int_C\omega\wed\la[\phi,A],\delbar_2\inv B\ra
			-\int_C\omega\wed\la[\phi,\delbar_1\inv A],B\ra
\end{align*}
so that (antisymmetrising the $A\leftrightarrow\phi$ terms)
\begin{align*}
	\int_C\omega\wed\la[\delbar\phi,\delbar_1\inv A],\delbar_2\inv B\ra
		- \int_C\omega\wed\la[A,\phi],\delbar_2\inv B\ra
		= -\int_C\omega\wed\la[\phi,\delbar_1\inv A],B\ra.
\end{align*}
Antisymmetrising this expression in the $A\leftrightarrow B$ terms gives
\begin{align*}
	\int_C\omega\wed\la[\delbar\phi,\delbar_1\inv A],\delbar_2\inv B\ra
		&- \int_C\omega\wed[A,\phi],\delbar_2\inv B\ra
		-\int_C\omega\wed\la[\delbar\phi,\delbar_1\inv B],\delbar_2\inv A\ra
		+ \int_C\omega\wed[B,\phi],\delbar_2\inv A\ra \\
	&= \int_C\omega\wed\la[\phi,\delbar_1\inv B],A\ra
		-\int_C\omega\wed\la[\phi,\delbar_1\inv A],B\ra \\
	&= \int_C\omega\wed\la[B,\delbar_1\inv A],\phi\ra
		-\int_C\omega\wed\la[A,\delbar_1\inv B],\phi\ra
\end{align*}
which is precisely the statement that the full antisymmetrisation of the original expression vanishes. Thus the expression is gauge invariant, and so descends to give a well-defined 3-form on $\mc{M}$.

Next, we would like to show that this 3-form is closed. Given a point $P\in\mc{M}$, identify a small neighbourhood $U$ of $P$ in $\mc{M}$ with a small neighbourhood of $0\in H^1(C;\mf{g}_P(-Q))$ by sending the harmonic representatives\footnote{Or any other choice of basis, as per our calculation of the derivative of the metric.} $\hat{A}\in\Omega^{0,1}(C;\mf{g}_P(-Q))$ for elements $A\in H^1(C;\mf{g}_P(-Q))$ to the $G$-bundles with $\delbar$-operator $\delbar_{\hat{A}}=\delbar_P+\hat{A}$. Over $U$, via this isomorphism, the tangent bundle is trivialised $TU\cong U\times H^1(C;\mf{g}_P(-Q))$ and the 3-form sends $\alpha_1,\alpha_2,\alpha_3\in H^1(C;\mf{g}_P(-Q))$ to the function on $U$ given by the antisymmetrization of
	\[	\int_C\omega\wed\la
		[\alpha_1,\delbar_{\hat{A},1}\inv(\alpha_2)],
		\delbar_{\hat{A},2}\inv(\alpha_3)
	\ra	\]
To calculate the exterior derivative of this function, we wish to antisymmetrize the terms which are $\alpha_4$-linear in the expression
\begin{align*}
	\int_C&\omega\wed\la
		[\alpha_1,\delbar_{\hat{A}+\alpha_4,1}\inv\alpha_2]\ra,
		\delbar_{\hat{A}+\alpha_4,2}\inv\alpha_3\ra
		-
		\int_C\omega\wed\la
		[\alpha_1,\delbar_{\hat{A},1}\inv\alpha_2],
		\delbar_{\hat{A},2}\inv\alpha_3\ra \\
		&=
		-\int_C\omega\wed\la
			\left[
				\alpha_1,
				\delbar_{\hat{A},1}\inv([\alpha_4,\delbar_{\hat{A},1}\inv\alpha_2])
			\right],
			\delbar_{\hat{A},2}\inv\alpha_3
		\ra
		-
		\int_C\omega\wed\la
			[\alpha_1,\delbar_{\hat{A},1}\inv\alpha_2],
			\delbar_{\hat{A},2}\inv([\alpha_4,\delbar_{\hat{A},2}\inv\alpha_3])
		\ra
		+ \text{ higher order terms} \\
		&=
		\int_C\omega\wed\la
				\delbar_{\hat{A},1}\inv([\alpha_4,\delbar_{\hat{A},1}\inv\alpha_2]),
			[\alpha_1,\delbar_{\hat{A},2}\inv\alpha_3]
		\ra
		-
		\int_C\omega\wed\la
			\delbar_{\hat{A},2}\inv([\alpha_4,\delbar_{\hat{A},2}\inv\alpha_3]),
			[\alpha_1,\delbar_{\hat{A},1}\inv\alpha_2]
		\ra
		+ \text{ h.o.t.}
\end{align*}
where we have used adjointness of $\delbar_1\inv$ and $-\delbar_2\inv$ (Lemma \ref{l:adjoints} below).

$\omega$ is a holomorphic section of $K_C(-D_1-D_2)$, so we have that
\begin{align*}
	\delbar &\left(\omega\tens\la
				\delbar_{\hat{A},1}\inv([\alpha_4,\delbar_{\hat{A},1}\inv\alpha_2]),
			\delbar_{\hat{A},2}([\alpha_1,\delbar_{\hat{A},2}\inv\alpha_3])
		\ra
		\right) \\
		&= -\omega\wed\la
				[\alpha_4,\delbar_{\hat{A},1}\inv\alpha_2],
			\delbar_{\hat{A},2}([\alpha_1,\delbar_{\hat{A},2}\inv\alpha_3])
		\ra
		-\omega\wed\la
				\delbar_{\hat{A},1}\inv([\alpha_4,\delbar_{\hat{A},1}\inv\alpha_2]),
			[\alpha_1,\delbar_{\hat{A},2}\inv\alpha_3]
		\ra \\
		&= -\omega\wed\la
			\delbar_{\hat{A},2}([\alpha_1,\delbar_{\hat{A},2}\inv\alpha_3]),
			[\alpha_4,\delbar_{\hat{A},1}\inv\alpha_2]
		\ra
		-\omega\wed\la
				\delbar_{\hat{A},1}\inv([\alpha_4,\delbar_{\hat{A},1}\inv\alpha_2]),
			[\alpha_1,\delbar_{\hat{A},2}\inv\alpha_3]
		\ra
\end{align*}

So the term linear in $\alpha_4$ becomes
	\[	\int_C\omega\wed\la
				\delbar_{\hat{A},1}\inv([\alpha_4,\delbar_{\hat{A},1}\inv\alpha_2]),
			[\alpha_1,\delbar_{\hat{A},2}\inv\alpha_3]
		\ra
		+
		\int_C\omega\wed\la
				\delbar_{\hat{A},1}\inv([\alpha_1,\delbar_{\hat{A},1}\inv\alpha_2]),
			[\alpha_4,\delbar_{\hat{A},2}\inv\alpha_3]
		\ra	\]
which is symmetric under the exchange of 1 and 4. Since the symmetric group $S_4$ can be written as a disjoint union $S_4 = C\coprod C\cdot (14)$ this expression vanishes upon antisymmetrising. Hence, the 3-form is closed.
\end{proof}

\begin{lemma}\label{l:adjoints}
For $A,B\in\Omega^{0,1}_C(\mf{g}_P(-Q))$, $\int_C\omega\wed\la\delbar_1\inv A, B\ra = -\int_C\omega\wed\la A,\delbar_2\inv B\ra$.
\end{lemma}
\begin{proof}
Observe that $\omega\la\delbar_1\inv A,\delbar_2\inv B\ra$ is holomorphic. So
	\[	0 = -\int_C\delbar\left(
		\omega\la\delbar_1\inv A,\delbar_2\inv B\ra
	\right)
			= \int_C\omega\wed\la A,\delbar_2\inv B\ra
				+\int_C\omega\wed\la \delbar_1\inv A,B\ra.	\]
\end{proof}

\begin{remark}\label{rk:3form-cech}
As for the metric, the above analysis of the 3-form can also be performed in the \v{C}ech model.
\end{remark}


%

%

%% file: section-action.tex

\section{The action and its variation}\label{s:action-variation}

Let us now consider the action of the two-dimensional sigma model built from the metric \eqref{eq:g} and 3-form \eqref{eq:3form}, namely\footnote{The factor of $\frac{1}{3}$ could be absorbed into the definition of the 3-form; leaving it explicit here will simplify some formulae later.}
\begin{align}\label{eq:action}
	S[\sigma]
		&= \int_{\mb{D}^2} \|d\sigma\|^2 d\vol_{\mb{D}^2}
			+ \frac{1}{3}\int_{\mb{D}^2\times\mb{R}_{\geq 0}} \tilde{\sigma}^\ast(\Omega)
\end{align}
where $\mb{D}^2$ is a two-dimensional disc and $\tilde\sigma$ is an extension of $\sigma$ to $\mb{D}^2\times\mb{R}_{\geq0}$.\footnote{Since $\Omega$ is closed, the classical equations of motion are insensitive to the choice of extension.}

Call the first term of \eqref{eq:action} $S_g$ and the second term $S_\Omega$. We will examine the variation of each of these terms. Let $\eta$ be a constant metric on the disc $\mb{D}^2$ (we will eventually take this to be the metric $dt_1dt_2$).

Since our spacetime is the disc we can work in coordinates $(\lambda^{ia})$ on the target space -- all of our equations will be derived using these coordinates.

\begin{prop}\label{p:vary-Sg}
Let $\Sigma(t_1,t_2,\tau)$ be a family of maps from $\mb{D}^2$ to $\mc{M}$ with $\Sigma(t_1,t_2,0)=\sigma$. Then
	\[	\left.\frac{d}{d\tau}\right|_{\tau=0} S_g[\Sigma]
			= \int_{\mb{D}^2}\left.\frac{\pd\Sigma^{kc}}{\pd\tau}\right|_{\tau=0}
				\eta^{\alpha\beta}\left(
					-2\sigma^\ast(g_{ia,kc})\frac{\pd^2\sigma^{ia}}{\pd t_\alpha \pd t_\beta}
					- \sigma^\ast\left(
						\frac{\pd g_{kc,jb}}{\pd\lambda^{ia}}
						+\frac{\pd g_{ia,kc}}{\pd\lambda^{jb}}
						-\frac{\pd g_{ia,jb}}{\pd\lambda^{kc}}
					\right)
					\frac{\pd\sigma^{ia}}{\pd t_\alpha}
					\frac{\pd\sigma^{jb}}{\pd t_\beta}
				\right)
				dt_1\wed dt_2.	\]
\end{prop}
\begin{proof}
A standard exercise in the calculus of variations.
\end{proof}

\begin{prop}\label{p:vary-Somega}
With notation as in Proposition \ref{p:vary-Sg}, the variation of $S_\Omega$ is
	\[	\left.\frac{d}{d\tau}\right|_{\tau=0} S_\Omega[\Sigma]
			= \int_{\mb{D}^2}\left.\frac{\pd\Sigma^{kc}}{\pd\tau}\right|_{\tau=0} \sigma^\ast(\Omega_{kc,ia,jb})\epsilon^{\alpha\beta}
			\frac{\pd\sigma^{ia}}{\pd t_\alpha}
			\frac{\pd\sigma^{jb}}{\pd t_\beta}
			dt_1\wed dt_2.	\]
\end{prop}

\begin{proof}
Let $\tilde\Sigma$ be an extension of the family of maps $\Sigma$ from Proposition \ref{p:vary-Sg} to $\mb{D}^2\times\mb{R}_{\geq0}$. Since $\Omega$ is closed,
	\[	\int_{\pd(\mb{D}^2\times \mb{R}_{\geq 0}\times [0,T])}
		\tilde\Sigma^\ast(\Omega)
		= 0
		= \int_{\mb{D}^2\times\mb{R}_{\geq0}}\tilde\Sigma_T^\ast(\Omega)
			-\int_{\mb{D}^2\times\mb{R}_{\geq0}}\tilde\sigma^\ast(\Omega)
			-\int_{\mb{D}^2\times[0,T]}\Sigma^\ast\Omega.	\]
So,
\begin{align*}
	\int_{\mb{D}^2\times\mb{R}_{\geq0}}\tilde\Sigma_T^\ast(\Omega)
			-\int_{\mb{D}^2\times\mb{R}_{\geq0}}\tilde\sigma^\ast(\Omega)
		&= \int_{\mb{D}^2\times[0,T]}\Sigma^\ast(\Omega_{ia,jb,kc})d\Sigma^{ia}\wed d\Sigma^{jb}\wed d\Sigma^{kc} \\
		&= 3\int_{\mb{D}^2\times[0,T]}\frac{\pd\Sigma^{kc}}{\pd\tau}\Sigma^\ast(\Omega_{kc,ia,jb})d_{\mb{D}}\Sigma^{ia}\wed d_{\mb{D}}\Sigma^{jb}\wed d\tau
\end{align*}
where $d_{\mb{D}}$ is the de Rham differential on $\mb{D}^2$. The result now follows by taking $\left.\frac{d}{dT}\right|_{T=0}$ and dividing by 3.
\end{proof}

Putting these two variational formulae together, we obtain:

\begin{cor}\label{cor:EOM}
The equations of motion for the action $S[\sigma]$ are
	\[	\sigma^\ast(g_{ia,kc})\eta^{\alpha\beta}\frac{\pd^2\sigma^{ia}}{\pd t_\alpha \pd t_\beta}
		+
		\frac{1}{2}\sigma^\ast\left(
			\frac{\pd g_{kc,jb}}{\pd\lambda^{ia}}
				+\frac{\pd g_{ia,kc}}{\pd\lambda^{jb}}
				-\frac{\pd g_{ia,jb}}{\pd\lambda^{kc}}
		\right)
		\eta^{\alpha\beta}
		\frac{\pd\sigma^{ia}}{\pd t_\alpha}
		\frac{\pd\sigma^{jb}}{\pd t_\beta}
		-
		\frac{1}{2}\sigma^\ast(\Omega_{ia,jb,kc})
		\epsilon^{\alpha\beta}
		\frac{\pd\sigma^{ia}}{\pd t_\alpha}
		\frac{\pd\sigma^{jb}}{\pd t_\beta}
		=
		0,	\]
or in expanded form for the metric $dt_1 dt_2$,
\begin{align*}
	2\sigma^\ast(g_{ia,kc})\frac{\pd^2\sigma^{ia}}{\pd t_1 \pd t_2}
		+
		\frac{1}{2}\sigma^\ast\left(
			\frac{\pd g_{kc,jb}}{\pd\lambda^{ia}}
				+\frac{\pd g_{ia,kc}}{\pd\lambda^{jb}}
				-\frac{\pd g_{ia,jb}}{\pd\lambda^{kc}}
		\right)
		&\left(
			\frac{\pd\sigma^{ia}}{\pd t_1}
			\frac{\pd\sigma^{jb}}{\pd t_2}
			+
			\frac{\pd\sigma^{ia}}{\pd t_2}
			\frac{\pd\sigma^{jb}}{\pd t_1}
		\right) \\
		&-
		\frac{1}{2}\sigma^\ast(\Omega_{ia,jb,kc})
		\left(
			\frac{\pd\sigma^{ia}}{\pd t_1}
			\frac{\pd\sigma^{jb}}{\pd t_2}
			-
			\frac{\pd\sigma^{ia}}{\pd t_2}
			\frac{\pd\sigma^{jb}}{\pd t_1}
		\right)
		=
		0.
\end{align*}
\end{cor}


%

%

%% file: section-cxn-target.tex

\section{Connections on the target space}\label{s:cxn-target}

In this section I will show that there are two natural families of algebraic connections on the target space $\mc{M}$, parametrised by $C_0$, and each related to one of the divisors $D_1$ or $D_2$.


\subsection{What is a connection?}\label{ss:what-is-cxn}

Let $\pi_{\mc{M}}:C\times\mc{M}\to\mc{M}$ be the projection. Then by the defining condition of $\mc{M}$,
	\[	R\pi_{\mc{M\ast}}(\ad(\mc{P})\tens\pi_C^\ast\mc{O}(D_i)) = 0,	\]
where $\pi_C:C\times\mc{M}\to C$ is the projection to $C$.

Consider $\Delta_{\mc{M}}:\mc{M}\to\mc{M}\times\mc{M}$, the diagonal map. This embedding corresponds to an ideal sheaf $\mc{I}$ on $\mc{M}\times\mc{M}$, and we can consider the quasicoherent sheaf of algebras on $\mc{M}$
	\[	0\to\Omega_{\mc{M}}^1
			\to\underbrace{p_{1\ast}\left(\mc{O}_{\mc{M}\times\mc{M}}/\mc{I}^2\right)}_{J^1(\mc{O}_{\mc{M}})}
			\to\mc{O}_{\mc{M}}
			\to 0	\]
where $p_1:\mc{M}\times\mc{M}\to\mc{M}$ is the first projection (and similarly for $p_2$). Note that $J^1(\mc{O}_{\mc{M}}) = p_{1\ast}p_2^\ast(\mc{O}_{\mc{M}})$.

Set $\Delta^{(1)}:=\rSpec_{\mc{M}\times\mc{M}}(\mc{O}_{\mc{M}\times\mc{M}}/\mc{I}^2)\subset\mc{M}\times\mc{M}$. Then a connection on an object\footnote{This is deliberately a bit vague: one can consider connections on objects in any Grothendieck fibration over our category of spaces, see \cite{nlab:grothendieck_connection}. We will not need this generality, however---blessedly, we restrict ourselves to vector bundles.} $B\to\mc{M}$ is an isomorphism $\phi:(p_1^\ast B)|_{\Delta^{(1)}}\simeq(p_2^\ast B)|_{\Delta^{(1)}}$ which restricts to the identity on $\mc{M}$.


\subsection{The connection of interest}\label{ss:cxn-of-interest}

Let $U=\Spec(R)\subset\mc{M}$ be an affine patch and consider trying to define a connection on $\mc{P}|_{C_0\times U}$ relative to $C_0$. Start by taking
	\[	\mc{P}|_{C\times U}\to C\times U	\]
and consider a square-zero extension
	\[	0\to J\to R' \to R\to 0,
		\qquad U':=\Spec(R').	\]
Equivalence classes of lifts of $\mc{P}|_{C\times U}$ together with its trivialisation on $Q$ to $C\times U'$ are parametrised by
	\[	R^1\pi_{U\ast}(\ad(\mc{P}|_{C\times U})(-Q))\tens_R J.	\]
Set $\supp(D_i)=\{x_{i,1},\ldots,x_{i,n}\}$ and let $\mb{D}_{i.j}$ be the formal disc around the point $x_{i,j}$ (similarly $\mb{D}_{i,j}^\times$ for the punctured formal disc). Consider the cover of $C$ given by
	\[	\mc{U}_i = \left\{ C_i, \bigcup_{j=1}^n\mb{D}_{i,j} \right\}.	\]
The above cohomology group can be calculated using the \v{C}ech complex for this cover:
\[\begin{tikzcd}
	\check{C}^0(\mc{U}_i;\ad(\mc{P})(-Q))\ar[d,equals]\ar[r,"\delta"]
		& \check{C}^1(\mc{U}_i;\ad(\mc{P})(-Q))\ar[d,equals] \\
	\Gamma(C_i;\ad(\mc{P})(-Q))\oplus
	\bigoplus_1^n\Gamma(\mb{D}_{i,j};\ad(\mc{P}))
		& \bigoplus_1^n\Gamma(\mb{D}^\times_{i,j};\ad(\mc{P}))
\end{tikzcd}\]
Consider the exact sequence on $C$
	\[	0\to \mc{O}_C(-Q) \to \mc{O}_C(D_i)\to
		\underbrace{\mc{O}_{D_i}(D_i)}_{\bigoplus_{j=1}^{g-1}T_{x_{i,j}}C}
		\to 0	\]
which yields the exact sequences on $C\times U$
	\[	0\to \ad(\mc{P}|_{C\times U})(-Q)
			\to \ad(\mc{P}|_{C\times U})\tens\pi_U^\ast\mc{O}_C(D_i)
			\to \underbrace{\ad(\mc{P}|_{C\times U})\tens \pi_U^\ast \mc{O}_{D_i}(D_i)}_{\bigoplus_1^{g-1}\ad(\mc{P}|_{x_{i,j}\times U})\tens T_{x_{i,j}}C}
			\to 0	\]
and so a sequence of \v{C}ech complexes with exact rows
\[\begin{tikzcd}[column sep = small]
	0\ar[r]
		& \check{C}^0(\mc{U}_i;\ad(\mc{P}|_{C\times U})(-Q)) \ar[r]\ar[d,"\delta"]
			& \check{C}^0(\mc{U}_i;\ad(\mc{P}|_{C\times U})\tens\pi_U^\ast\mc{O}_C(D_i)) \ar[r]\ar[d,"\sim"]
				& \bigoplus_1^n\ad(\mc{P}|_{x_{i,j}\times U})\tens T_{x_{i,j}}C\ar[d]\ar[r]
					& 0 \\
	0\ar[r]
		& \check{C}^1(\mc{U}_i;\ad(\mc{P}|_{C\times U})(-Q))\ar[r]\ar[ur,dashed]
			& \check{C}^1(\mc{U}_i;\ad(\mc{P}|_{C\times U})\tens\pi_U^\ast\mc{O}_C(D_i))\ar[r]
				& 0
\end{tikzcd}\]

The dashed arrow gives a map that takes a lift of $\ad(\mc{P}|_{C\times U})(-Q)$ to $C\times U'$ to a trivialisation of that lift \emph{with first-order poles on $D_i$}.

Given two lifts, $\tilde{\mc{P}}$, $\tilde{\mc{P}}'$, taking their difference under the dashed arrow gives a unique isomorphism
	\[	\tilde{\mc{P}}|_{C_i\times U}\simeq\tilde{\mc{P}}'|_{C_i\times U}	\]
with first-order poles on $D_i$.

So given the two pulllbacks of $\ad(\mc{P}|_{C_0\times U})(-Q)$ to the square-zero neighbourhood, we obtain a unique isomorphism between them with first-order poles on $D_i$, call this $\phi_{i,U}$.

The above diagram is compatible with restricting the affine open $V\subset U$, and the corresponding maps $\phi_{i,V}$ are unique; thus the collection $\{\phi_{i,U}\}_{U\subset\mc{M}}$ glues to give a connection $\phi_i$ on $\ad(\mc{P}|_{C_0\times\mc{M}})(-Q)$ relative to $C_0\times\mc{M}\to C_0$.

Denote the connection associated to the divisor $D_1$ by $\gd^+$ and the connection associated to the divisor $D_2$ by $\gd^-$. The diagram above also tells us how to derive the connection 1-form in local coordinates $(\lambda^{ia})$ centred at $P\in\mc{M}$: writing $D_1= d+\alpha= d+\alpha_{ia}d\lambda^{ia}$ and $D_2= d+\beta= d+\beta_{ia}d\lambda^{ia}$, the connection components are precisely given by the singular gauge trivialisations of the basis elements, i.e.\
\begin{alignat}{3}\label{eq:cxn-comps}
	\alpha_{ia}(\vec\lambda)
		&= \delbar_{P+\vec\lambda,1}\inv A_{ia},
		&\qquad
	&\beta_{ia}(\vec\lambda)
		&= \delbar_{P+\vec\lambda,2}\inv A_{ia}.
\end{alignat}

%% file: section-cxn-induced.tex
\section{The induced connection and flatness equation}\label{s:cxn-induced}

Now, given the connections $\gd^\pm$ and a field $\sigma:\mb{D}^2\to\mc{M}$, we wish to define an induced connection $D(\sigma)$ on $\mb{D}^2$ as follows. Recall that we have coordinates $t_1,t_2$ on $\mb{D}^2$, which are the null directions for our metric $dt_1 dt_2$, and set $\pd_\alpha = \frac{\pd}{\pd t_\alpha}$.

\begin{defn}\label{def:cxn-induced}
The induced connection $D(\sigma)$ is defined by the equations
\begin{alignat}{3}
	D(\sigma)_{\pd_1} &:= (\sigma^\ast\gd^+)_{\pd_1},
	& \qquad
	& D(\sigma)_{\pd_2} &:= (\sigma^\ast\gd^-)_{\pd_2}.
\end{alignat}
\end{defn}

In terms of the connection forms on the target we can write these equations as
\begin{align*}
	D(\sigma)_{\pd_1} &:= \frac{\pd}{\pd t_1} + \left[\sigma^\ast\alpha\left(\frac{\pd}{\pd t_1}\right),-\right],
	& D(\sigma)_{\pd_2} &:= \frac{\pd}{\pd t_2} + \left[\sigma^\ast\beta\left(\frac{\pd}{\pd t_2}\right),-\right].
\end{align*}
so that in local coordinates we can write
	\[	D(\sigma) = d + \sigma^\ast\alpha\left(\frac{\pd}{\pd t_1}\right)dt_1 + \sigma^\ast\beta\left(\frac{\pd}{\pd t_2}\right)dt_2.	\]

\begin{prop}\label{p:flatness}
$D(\sigma)$ is flat if and only if
\begin{align}\label{eq:flatness}
	\sigma^\ast(\beta_{ia}-\alpha_{ia})\frac{\pd^2\sigma^{ia}}{\pd t_1 \pd t_2}
	+
	\sigma^\ast\left(
		\frac{\pd\beta_{jb}}{\pd\lambda^{ia}}
		-
		\frac{\pd\alpha_{ia}}{\pd\lambda^{jb}}
		+
		[\alpha_{ia},\beta_{jb}]
	\right)
	\frac{\pd\sigma^{ia}}{\pd t_1}
	\frac{\pd\sigma^{jb}}{\pd t_2}
	=
	0.
\end{align}
\end{prop}
\begin{proof}
$D(\sigma)$ has curvature
	\[	\frac{F(\sigma)}{dt_1\wed dt_2}
			=
			\frac{\pd}{\pd t_1}\left(
					\sigma^\ast\beta\left(\frac{\pd}{\pd t_2}\right)
				\right)
			-
			\frac{\pd}{\pd t_2}\left(
					\sigma^\ast\alpha\left(\frac{\pd}{\pd t_1}\right)
				\right)
			+
			\left[
				\sigma^\ast\alpha\left(\frac{\pd}{\pd t_1}\right),
				\sigma^\ast\beta\left(\frac{\pd}{\pd t_2}\right)
			\right].	\]
In local coordinates $(\lambda^{ia})$ near $P\in\mc{M}$ we write
\begin{alignat}{5}
	\alpha &= \alpha_{ia}d\lambda^{ia},
	& \qquad
	&\beta &= \beta_{ia}d\lambda^{ia},
	& \qquad
	&\sigma&= (\sigma^{ia}(t_1,t_2)),
\end{alignat}
so that
\begin{align*}
	\sigma^\ast\alpha
		&= \alpha_{ia}(\sigma(t_1,t_2))d\sigma^{ia}
			= \sigma^\ast(\alpha_{ia})\frac{\pd\sigma^{ia}}{\pd t_\alpha} dt_\alpha,
\end{align*}
and similarly for $\sigma^\ast\beta$. Then
\begin{align*}
	\frac{\pd}{\pd t_1}\left(
		\sigma^\ast\beta\left(\frac{\pd}{\pd t_2}\right)
	\right)
		&= \frac{\pd}{\pd t_1}\left(
				\beta_{ia}(\sigma(t_1,t_2))\frac{\pd\sigma^{ia}}{\pd t_2}
			\right) \\
		&= \frac{\pd\beta_{ia}}{\pd\sigma^{jb}}
			\frac{\pd\sigma^{jb}}{\pd t_1}
			\frac{\pd\sigma^{ia}}{\pd t_2}
			+
			\beta_{ia}(\sigma(t_1,t_2))\frac{\pd^2\sigma^{ia}}{\pd t_1 \pd t_2} \\
		&= \sigma^\ast\left(\frac{\pd\beta_{ia}}{\pd\lambda^{jb}}\right)
		\frac{\pd\sigma^{jb}}{\pd t_1}\frac{\pd\sigma^{ia}}{\pd t_2}
		+
		\sigma^\ast(\beta_{ia})\frac{\pd^2\sigma^{ia}}{\pd t_1 \pd t_2}
\end{align*}
and similarly for $\alpha$, so that
\begin{align*}
	\frac{F(\sigma)}{dt_1\wed dt_2}
			&=
			\frac{\pd}{\pd t_1}\left(
					\sigma^\ast\beta\left(\frac{\pd}{\pd t_2}\right)
				\right)
			-
			\frac{\pd}{\pd t_2}\left(
					\sigma^\ast\alpha\left(\frac{\pd}{\pd t_1}\right)
				\right)
			+
			\left[
				\sigma^\ast\alpha\left(\frac{\pd}{\pd t_1}\right),
				\sigma^\ast\beta\left(\frac{\pd}{\pd t_2}\right)
			\right] \\
			&=
			\sigma^\ast\left(\frac{\pd\beta_{ia}}{\pd\lambda^{jb}}\right)
			\frac{\pd\sigma^{jb}}{\pd t_1}
			\frac{\pd\sigma^{ia}}{\pd t_2}
			+
			\sigma^\ast\beta_{ia}\frac{\pd^2\sigma^{ia}}{\pd t_1 \pd t_2} \\
			&\mspace{100mu}
			-
			\sigma^\ast\left(\frac{\pd\alpha_{ia}}{\pd\lambda^{jb}}\right)
			\frac{\pd\sigma^{jb}}{\pd t_2}
			\frac{\pd\sigma^{ia}}{\pd t_1}
			-
			\sigma^\ast\alpha_{ia}\frac{\pd^2\sigma^{ia}}{\pd t_1 \pd t_2}
			+
			\left[
				\sigma^\ast\alpha_{ia}
				\frac{\pd\sigma^{ia}}{\pd t_1},
				\sigma^\ast\beta_{jb}
				\frac{\pd\sigma^{jb}}{\pd t_2}
			\right] \\
			&=
			\sigma^\ast(\beta_{ia}-\alpha_{ia})
			\frac{\pd^2\sigma^{ia}}{\pd t_1 \pd t_2}
			+
			\sigma^\ast\left(
				\frac{\pd\beta_{jb}}{\pd\lambda^{ia}}
				-
				\frac{\pd\alpha_{ia}}{\pd\lambda^{jb}}
				+
				[\alpha_{ia},\beta_{jb}]
			\right)
			\frac{\pd\sigma^{ia}}{\pd t_1}\frac{\pd\sigma^{jb}}{\pd t_2}.
\end{align*}
\end{proof}

I am soon going to want to compare the flatness equation to the equations of motion; for this reason, it is useful to ``re-symmetrise'' the flatness equation to place it in a form that more closely resembles the desired result. Namely, we take the expression \eqref{eq:flatness} for $F(\sigma)$ and rewrite it as
\begin{align*}
	\sigma^\ast(\beta_{ia}-\alpha_{ia})
	&\frac{\pd^2\sigma^{ia}}{\pd t_1 \pd t_2}
	+
	\sigma^\ast\left(
		\frac{\pd\beta_{jb}}{\pd\lambda^{ia}}
		-
		\frac{\pd\alpha_{ia}}{\pd\lambda^{jb}}
		+
		[\alpha_{ia},\beta_{jb}]
	\right)
	\frac{\pd\sigma^{ia}}{\pd t_1}
	\frac{\pd\sigma^{jb}}{\pd t_2}	\\
	&=
		\sigma^\ast(\beta_{ia}-\alpha_{ia})\frac{\pd^2\sigma^{ia}}{\pd t_1 \pd t_2} \\
	& \qquad
		+
		\sigma^\ast\left(
		\frac{\pd\beta_{jb}}{\pd\lambda^{ia}}
		-
		\frac{\pd\alpha_{ia}}{\pd\lambda^{jb}}
		+
		[\alpha_{ia},\beta_{jb}]
	\right)
	\left(
		\frac{\pd\sigma^{ia}}{\pd t_1}
		\frac{\pd\sigma^{jb}}{\pd t_2}
		+
		\frac{1}{2}
		\frac{\pd\sigma^{jb}}{\pd t_1}
		\frac{\pd\sigma^{ia}}{\pd t_2}
		-
		\frac{1}{2}
		\frac{\pd\sigma^{jb}}{\pd t_1}
		\frac{\pd\sigma^{ia}}{\pd t_2}
	\right) \\
	&=
		\sigma^\ast(\beta_{ia}-\alpha_{ia})\frac{\pd^2\sigma^{ia}}{\pd t_1 \pd t_2} \\
	& \qquad
		+
		\frac{1}{2}
		\sigma^\ast\left(
		\frac{\pd\beta_{jb}}{\pd\lambda^{ia}}
		-
		\frac{\pd\alpha_{ia}}{\pd\lambda^{jb}}
		+
		[\alpha_{ia},\beta_{jb}]
	\right)
	\left(
		\frac{\pd\sigma^{ia}}{\pd t_1}
		\frac{\pd\sigma^{jb}}{\pd t_2}
		+
		\frac{\pd\sigma^{jb}}{\pd t_1}
		\frac{\pd\sigma^{ia}}{\pd t_2}
	\right) \\
	& \qquad
		+
		\frac{1}{2}
		\sigma^\ast\left(
		\frac{\pd\beta_{jb}}{\pd\lambda^{ia}}
		-
		\frac{\pd\alpha_{ia}}{\pd\lambda^{jb}}
		+
		[\alpha_{ia},\beta_{jb}]
	\right)
	\left(
		\frac{\pd\sigma^{ia}}{\pd t_1}
		\frac{\pd\sigma^{jb}}{\pd t_2}
		-
		\frac{\pd\sigma^{jb}}{\pd t_1}
		\frac{\pd\sigma^{ia}}{\pd t_2}
	\right) \\
	&=
		\sigma^\ast(\beta_{ia}-\alpha_{ia})\frac{\pd^2\sigma^{ia}}{\pd t_1 \pd t_2} \\
	& \qquad
		+
		\frac{1}{4}
		\sigma^\ast\left(
		\frac{\pd\beta_{jb}}{\pd\lambda^{ia}}
		-
		\frac{\pd\alpha_{ia}}{\pd\lambda^{jb}}
		+
		[\alpha_{ia},\beta_{jb}]
		+
		\frac{\pd\beta_{ia}}{\pd\lambda^{jb}}
		-
		\frac{\pd\alpha_{jb}}{\pd\lambda^{ia}}
		+
		[\alpha_{jb},\beta_{ia}]
	\right)
	\left(
		\frac{\pd\sigma^{ia}}{\pd t_1}
		\frac{\pd\sigma^{jb}}{\pd t_2}
		+
		\frac{\pd\sigma^{jb}}{\pd t_1}
		\frac{\pd\sigma^{ia}}{\pd t_2}
	\right) \\
	& \qquad
		+
		\frac{1}{4}
		\sigma^\ast\left(
		\frac{\pd\beta_{jb}}{\pd\lambda^{ia}}
		-
		\frac{\pd\alpha_{ia}}{\pd\lambda^{jb}}
		+
		[\alpha_{ia},\beta_{jb}]
		-
		\frac{\pd\beta_{ia}}{\pd\lambda^{jb}}
		+
		\frac{\pd\alpha_{jb}}{\pd\lambda^{ia}}
		-
		[\alpha_{jb},\beta_{ia}]
	\right)
	\left(
		\frac{\pd\sigma^{ia}}{\pd t_1}
		\frac{\pd\sigma^{jb}}{\pd t_2}
		-
		\frac{\pd\sigma^{jb}}{\pd t_1}
		\frac{\pd\sigma^{ia}}{\pd t_2}
	\right)
\end{align*}


%

%

%% file: section-main-theorem.tex

\section{Flatness and the equations of motion}\label{s:main-theorem}

Given a field $\sigma:\mb{D}^2 \to \mc{M}$, there are two objects
we can now consider:
\begin{itemize}
    \item The action \eqref{eq:action} and corresponding equations
            of motion (Cor.\ \ref{cor:EOM}).
    \item The induced connection on spacetime, $D(\sigma)$ (Def.\ \ref{def:cxn-induced}).
\end{itemize}

The following is the main theorem of this paper:

\begin{thm}
The field $\sigma$ satisfies the equations of motion for $S[\sigma]$ if and only if the connection $D(\sigma)$ is flat.
\end{thm}

\begin{proof}
We proceed by comparing the coefficients of $\frac{\pd^2\sigma^{ia}}{\pd t_1 \pd t_2}$, $\frac{\pd\sigma^{ia}}{\pd t_1}\frac{\pd\sigma^{jb}}{\pd t_2} + \frac{\pd\sigma^{jb}}{\pd t_1}\frac{\pd\sigma^{ia}}{\pd t_2}$, and $\frac{\pd\sigma^{ia}}{\pd t_1}\frac{\pd\sigma^{jb}}{\pd t_2} - \frac{\pd\sigma^{jb}}{\pd t_1}\frac{\pd\sigma^{ia}}{\pd t_2}$ in the equations of motion and flatness equations. We see that we wish to compare
\begin{align}
    \text{Equations of motion}  && \text{Flatness} \nonumber \\
    2 g_{ia,kc} && \beta_{ia} - \alpha_{ia} \label{eq:order2term}\\
    \frac{1}{2}\left(
            \frac{\pd g_{kc,jb}}{\pd\lambda^{ia}}
                +\frac{\pd g_{ia,kc}}{\pd\lambda^{jb}}
                -\frac{\pd g_{ia,jb}}{\pd\lambda^{kc}}
        \right)
        && \frac{1}{4}\left(
        \frac{\pd\beta_{jb}}{\pd\lambda^{ia}}
        -
        \frac{\pd\alpha_{ia}}{\pd\lambda^{jb}}
        +
        [\alpha_{ia},\beta_{jb}]
        +
        \frac{\pd\beta_{ia}}{\pd\lambda^{jb}}
        -
        \frac{\pd\alpha_{jb}}{\pd\lambda^{ia}}
        +
        [\alpha_{jb},\beta_{ia}]
        \right) \label{eq:symterm}\\
    -\frac{1}{2}\Omega_{ia,jb,kc}
        && \frac{1}{4}\left(
        \frac{\pd\beta_{jb}}{\pd\lambda^{ia}}
        -
        \frac{\pd\alpha_{ia}}{\pd\lambda^{jb}}
        +
        [\alpha_{ia},\beta_{jb}]
        -
        \frac{\pd\beta_{ia}}{\pd\lambda^{jb}}
        +
        \frac{\pd\alpha_{jb}}{\pd\lambda^{ia}}
        -
        [\alpha_{jb},\beta_{ia}]
        \right) \label{eq:altterm}
\end{align}
Note that, by the same calculations as we performed for the derivative of the metric,
\begin{align*}
    \left.\frac{\pd\alpha_{ia}}{\pd\lambda^{jb}}\right|_{\vec\lambda =0 }
        &= - \delbar_{P,1}\inv [A_{jb},\delbar_{P,1}\inv A_{ia}],
    &\left.\frac{\pd\beta_{jb}}{\pd\lambda^{ia}}\right|_{\vec\lambda =0 }
        &= - \delbar_{P,2}\inv [A_{ia},\delbar_{P,2}\inv A_{jb}]
\end{align*}
Since the equations of motion and flatness equations have well-defined coordinate-free descriptions, to compare these equations it is enough to compare them at every point $P\in\mc{M}$ in coordinates centred at $P$. Thus from now on I will omit the $\left.\phantom{\frac{a}{b}}\right|_{\vec\lambda=0}$ symbol from my calculations -- all terms are to be implicitly evaluated at the origin of the coordinate system.

Notice that a priori there are many (infinitely!) more ``flatness equations'', since the equations of motion are indexed by $(ia,kc)$ while the flatness equations are indexed by $(ia,z\in C)$. To obtain a matching number of equations, we take the flatness equations and pair them with each basis vector $A_{kc}$, i.e.\ we take
    \[   \int_C\omega\wed\la - , A_{kc} \ra.   \]
Doing this for \eqref{eq:order2term} gives
\begin{align*}
    \int_C \omega \wed\la \beta_{ia} - \alpha_{ia}, A_{kc}\ra
        &= \int_C \omega\wed \la \delbar_2\inv A_{ia}, A_{kc}\ra
            - \int_C\omega\wed\la \delbar_1\inv A_{ia}, A_{kc}\ra \\
        &= - \int_C\omega\wed\la \delbar_1\inv A_{ia}, A_{kc}\ra
            - \int_C\omega\wed\la A_{ia}, \delbar_1\inv A_{kc}\ra \\
        &= - g_{ia,kc} = -\frac{1}{2}\left(2 g_{ia,kc}\right).
\end{align*}
Next, let's compare the terms in \eqref{eq:symterm}. Writing the equations of motion explicitly gives
\begin{align*}
    \frac{1}{2}&\left(
            \frac{\pd g_{kc,jb}}{\pd\lambda^{ia}}
                +\frac{\pd g_{ia,kc}}{\pd\lambda^{jb}}
                -\frac{\pd g_{ia,jb}}{\pd\lambda^{kc}}
        \right) \\
    & =
        \underbrace{\frac{1}{2}\int_C\omega\wed\la [A_{jb},\delbar_1\inv A_{ia}],\delbar_2\inv A_{kc}\ra}_{(I)}
        +\underbrace{\frac{1}{2}\int_C\omega\wed\la [A_{jb},\delbar_1\inv A_{kc}],\delbar_2\inv A_{ia}\ra}_{(II)} \\
    &\qquad
        +\underbrace{\frac{1}{2}\int_C\omega\wed\la [A_{ia},\delbar_1\inv A_{jb}],\delbar_2\inv A_{kc}\ra}_{(III)}
        +\underbrace{\frac{1}{2}\int_C\omega\wed\la [A_{ia},\delbar_1\inv A_{kc}],\delbar_2\inv A_{jb}\ra}_{(IV)} \\
    &\qquad
        -\underbrace{\frac{1}{2}\int_C\omega\wed\la [A_{kc},\delbar_1\inv A_{ia}],\delbar_2\inv A_{jb}\ra}_{(V)}
        -\underbrace{\frac{1}{2}\int_C\omega\wed\la [A_{kc},\delbar_1\inv A_{jb}],\delbar_2\inv A_{ia}\ra}_{(VI)},
\end{align*}
while the corresponding term for the flatness equations gives
\begin{align*}
    \int_C &\omega\wed\la
    \frac{1}{4}\left(
        \frac{\pd\beta_{jb}}{\pd\lambda^{ia}}
        -
        \frac{\pd\alpha_{ia}}{\pd\lambda^{jb}}
        +
        [\alpha_{ia},\beta_{jb}]
        +
        \frac{\pd\beta_{ia}}{\pd\lambda^{jb}}
        -
        \frac{\pd\alpha_{jb}}{\pd\lambda^{ia}}
        +
        [\alpha_{jb},\beta_{ia}]
        \right)
    , A_{kc}\ra \\
    &=
        \frac{1}{4}\int_C\omega\wed\left\langle
            -\delbar_2\inv[A_{ia},\delbar_2\inv A_{jb}]
            +\delbar_1\inv[A_{jb},\delbar_1\inv A_{ia}]
            +[\delbar_1\inv A_{ia},\delbar_2\inv A_{jb}]\right.\\
    &\mspace{100mu}\left.
            -\delbar_2\inv[A_{jb},\delbar_2\inv A_{ia}]
            +\delbar_1\inv[A_{ia},\delbar_1\inv A_{jb}]
            +[\delbar_1\inv A_{jb},\delbar_2\inv A_{ia}]
            ,
            A_{kc}
        \right\rangle \\
    &=
        -\frac{1}{4}\underbrace{\int_C\omega\wed\la
            [A_{ia},\delbar_1\inv A_{kc}], \delbar_2\inv A_{jb}
        \ra}_{(IV)}
        -\frac{1}{4}\underbrace{\int_C\omega\wed\la
            [A_{jb},\delbar_1\inv A_{ia}], \delbar_2\inv A_{kc}
        \ra}_{(I)} \\
    &\qquad
        +\frac{1}{4}\underbrace{\int_C\omega\wed\la
            [A_{kc},\delbar_1\inv A_{ia}], \delbar_2\inv A_{jb}
        \ra}_{(V)}
        -\frac{1}{4}\underbrace{\int_C\omega\wed\la
            [A_{jb},\delbar_1\inv A_{kc}], \delbar_2\inv A_{ia}
        \ra}_{(II)} \\
    &\qquad
        -\frac{1}{4}\underbrace{\int_C\omega\wed\la
            [A_{ia},\delbar_1\inv A_{jb}], \delbar_2\inv A_{kc}
        \ra}_{(III)}
        +\frac{1}{4}\underbrace{\int_C\omega\wed\la
            [A_{kc},\delbar_1\inv A_{jb}], \delbar_2\inv A_{ia}
        \ra}_{(VI)} \\
    &= -\frac{1}{2}\left[
        \frac{1}{2}\left(
            \frac{\pd g_{kc,jb}}{\pd\lambda^{ia}}
                +\frac{\pd g_{ia,kc}}{\pd\lambda^{jb}}
                -\frac{\pd g_{ia,jb}}{\pd\lambda^{kc}}
        \right)
    \right]
\end{align*}
where we have freely made use of the adjointness of $\delbar_1\inv$ and $-\delbar_2\inv$, and the underbraced indices indicate how to match the flatness and equation of motion terms.

Finally, consider the flatness expression in \eqref{eq:altterm}:
\begin{align*}
    \int_C &\omega\wed\la
    \frac{1}{4}\left(
        \frac{\pd\beta_{jb}}{\pd\lambda^{ia}}
        -
        \frac{\pd\alpha_{ia}}{\pd\lambda^{jb}}
        +
        [\alpha_{ia},\beta_{jb}]
        -
        \frac{\pd\beta_{ia}}{\pd\lambda^{jb}}
        +
        \frac{\pd\alpha_{jb}}{\pd\lambda^{ia}}
        -
        [\alpha_{jb},\beta_{ia}]
        \right)
    , A_{kc}\ra \\
    &=
        \frac{1}{4}\int_C\omega\wed\left\langle
            -\delbar_2\inv[A_{ia},\delbar_2\inv A_{jb}]
            +\delbar_1\inv[A_{jb},\delbar_1\inv A_{ia}]
            +[\delbar_1\inv A_{ia},\delbar_2\inv A_{jb}]\right.\\
    &\mspace{100mu}\left.
            +\delbar_2\inv[A_{jb},\delbar_2\inv A_{ia}]
            -\delbar_1\inv[A_{ia},\delbar_1\inv A_{jb}]
            -[\delbar_1\inv A_{jb},\delbar_2\inv A_{ia}]
            ,
            A_{kc}
        \right\rangle \\
    &= \frac{1}{4}\int_C \omega\wed\underbrace{
        \la
        [A_{ia},\delbar_2]\inv A_{jb}], \delbar_1\inv A_{kc}
        \ra
    }_{(23)}
        -\frac{1}{4}\int_C \omega\wed\underbrace{
        \la
        [A_{jb},\delbar_1\inv A_{ia}], \delbar_2\inv A_{kc}
        \ra
        }_{(12)} \\
    &\qquad
        +\frac{1}{4}\int_C \omega\wed\underbrace{
        \la
        [\delbar_1\inv A_{ia}, \delbar_2\inv A_{jb}], A_{kc}
        \ra
        }_{(123)}
        -\frac{1}{4}\int_C \omega\wed\underbrace{
        \la
        [A_{jb},\delbar_2\inv A_{ia}], \delbar_1\inv A_{kc}
        \ra
        }_{(132)} \\
    &\qquad
        +\frac{1}{4}\int_C \omega\wed\underbrace{
        \la
        [A_{ia},\delbar_1\inv A_{jb}], \delbar_2\inv A_{kc}
        \ra
        }_{(1)}
        -\frac{1}{4}\int_C \omega\wed\underbrace{
        \la
        [\delbar_1\inv A_{jb}, \delbar_2\inv A_{ia}], A_{kc}
        \ra
        }_{(13)} \\
    &= \frac{1}{4}\sum_{s\in S_3} (-1)^s\int_C\omega\wed
            \la
            [A_{s(ia)}, \delbar_1\inv A_{s(jb)}], \delbar_2\inv A_{s(kc)}
            \ra \\
    &= -\frac{1}{2}\left(
            -\frac{1}{2}\Omega_{ia,jb,kc}
        \right)
\end{align*}
where the underbraced incides now indicate which term corresponds to which element of the symmetric group $S_3$. But now observe that we have
    \[   \int_C\omega\wed \la \text{Flatness}, A_{kc} \ra
            =
            -\frac{1}{2}\text{EOM}_{kc}   \]
for all basis elements $A_{kc}$, completing the proof of the theorem.
\end{proof}

%% file: section-CP1-eg.tex
\section{Example: $\mb{CP}^1$ with two marked points}\label{s:CP1-eg}

Let us now consider an example where we can write things out fairly explicitly: the case of $\mb{CP}^1$ with two marked points.

Let $C=\mb{CP}^1$ equipped with coordinate $z$ and 1-form
\begin{align}\label{eq:cp1-1form}
	\omega &= \frac{(z-p_1)(z-p_2)}{z^2}dz,\quad p_1\neq p_2,\, p_i\neq 0.
	\end{align}
Note that $\omega$ has second order poles at $0$ and $\infty$, and simple zeroes at $p_1$ and $p_2$. Let $Q$ denote the divisor $0 + \infty$. We wish to consider $\BunSt_G(\mb{CP}^1, Q)$, the moduli of $G$-bundles on $\mb{CP}^1$ with a trivialisation on $Q$.

Assume that $G$ is simple simply-connected, and consider the cohomology vanishing condition
	\[	H^0(\mb{CP}^1,\mf{g}_P(D_1))=0,	\]
where we choose $D_1 = p_1 - 0 - \infty$. In this example, the cohomology vanishing condition implies that the bundle $P$ is trivializable.\footnote{To see this, suppose $E=\mc{O}(i_1)\oplus\cdots\mc{O}(i_n)$ with $i_1 + \cdots + i_n = 0$. Then $E\tens\mc{O}(-1)$ has vanishing $H^0$ if and only if all of the $i_k<1$, which forces $i_1 = \cdots = i_n = 0$.} So, fix a trivialized bundle with trivialisations at $0$ and $\infty$, $P = \mb{CP}^1\times G$. We can always act on the bundle to simultaneously change the trivialisations at $0$ and $\infty$, so fix the trivialisation at $\infty$ to agree with the trivialisation of the bundle. Our remaining degree of freedom is the trivialisation of the bundle at $0$, and so we see that $\mc{M} \cong G$. As a check on this, observe that we have an agreement of tangent spaces
	\[	T_P\mc{M} = H^1(C;\mf{g}\tens\mc{O}(-0-\infty))
					\cong \mf{g}\tens H^1(C;\mc{O}(-2))
					\cong \mf{g}.	\]
Note that in this example we can explicitly determine the form of the Szeg\"o kernel. For the trivial $G$-bundle on $\mb{CP}^1$ this will be a $\mf{g}\tens\mf{g}$-valued meromorphic function $\mc{S}(z,z')$ on $\mb{CP}^1\times\mb{CP}^1$ with
\begin{itemize}
	\item simple pole along $z=z'$ with residue the Casimir element $c\in\mf{g}\tens\mf{g}$,
	\item simple poles at $z=p_1$ and $z'=p_2$, and
	\item simple zeroes at $z=0,\infty$, $z'=0,\infty$.
	\end{itemize}
Given these conditions, it is easy to check that the kernel is
\begin{align}\label{eq:szego-cp1}
	\mc{S}(z,z')
		&= \frac{zz'}{(z-z')(z-p_1)(z'-p_2)}c
\end{align}

Now, we want to calculate the metric and 3-form on $\mc{M}$. Since $\mc{M}$ has a transitive $G\times G$-symmetry, given by left and right $G$-multiplication on the trivialisation at $0$, and the expressions for the metric and 3-form are independent of the trivialisation at $\infty$, it suffices to calculate at the basepoint $\ast$ of $\mc{M}$ given by the trivial trivialisation at $\infty$; i.e.\ we are looking at multiples of the usual bi-invariant metric and 3-form on $G$.

Let $\{t_a\}$ be a basis for $\mf{g}$, and take as representatives for a basis of $T_\ast\mc{M}$ the forms
	\[	A_a = (p_1-p_2)\frac{z}{(z-p_1)(z-p_2)}\delta_{|z-p_1|=\epsilon}t_a	\]
These correspond to the basis $\{t_a\}$ under the identification of $\mc{M}$ with $G$, which we can see in the following way: a tangent vector to $\mc{M}$ is represented by a form $A\in\Omega^{0,1}(\mb{CP}^1;\mc{O}(-0-\infty))\tens\mf{g}$, and if it is nonzero it has no antiderivative in $\Omega^{0,0}(\mb{CP}^1;\mc{O}(-0-\infty))\tens\mf{g}$. It \emph{does} however have an antiderivative $\chi\in\Omega^{0,0}(\mb{CP}^1;\mc{O}(-\infty))\tens\mf{g}$---i.e.\ we allow the antiderivative to be nonzero at $0$---and then $\chi(0)\in\mf{g}$ is precisely the infinitesimal variation of the framing at $0$.

Observing that
	\[	(p_1-p_2)\frac{z}{(z-p_1)(z-p_2)}
			= \frac{p_1}{z-p_1} - \frac{p_2}{z-p_2}	\]
we have that $A_a = \delbar\chi_a$ for
	\[	\chi_a = -\left(
			\frac{p_1}{z-p_1}\delta_{|z-p_1|\geq\epsilon}
			+
			\frac{p_2}{z-p_2}\delta_{|z-p_1|\leq\epsilon}
	\right)t_a	\]
and so $\chi_a(0)=t_a$.

Note that these also have singular antiderivatives
\begin{align*}
	\delbar_1\inv A_a
		&= (p_1-p_2)\frac{z}{(z-p_1)(z-p_2)}\delta_{|z-p_1|\leq\epsilon}t_a,
	&\delbar_2\inv A_a
		&= -(p_1-p_2)\frac{z}{(z-p_1)(z-p_2)}\delta_{|z-p_1|\geq\epsilon}t_a,
\end{align*}
so that
\begin{align*}
	g_{ab}
		&= \int_{\mb{CP}^1}\omega\wed\la\delbar_1\inv A_a,A_b\ra
			+\int_{\mb{CP}^1}\omega\wed\la\delbar_1\inv A_b,A_a\ra
		= \kappa_{ab}\oint_{|z-p_1|=\epsilon}
				\frac{dz}{z-p_1}\frac{(p_1-p_2)^2}{z-p_2}
		= 2\pi i (p_1 - p_2)\kappa_{ab}
\end{align*}
and
\begin{align*}
	\int_{\mb{CP}^1}\omega \wed\la [A_a,\delbar_1\inv A_b], \delbar_2\inv A_c\ra
		&= \frac{1}{3}\oint_{|z-p_1|=\epsilon}\frac{dz}{(z-p_1)^2}\frac{z}{(z-p_2)^2}\la[t_a,t_b],t_c\ra (p_1 - p_2)^3 \\
		&= \frac{2\pi i}{3}\kappa_{cd}f_{ab}^d \frac{d}{dz}\left.\left(\frac{z}{(z-p_2)^2}\right)\right|_{z=p_1}\cdot (p_1 - p_2)^3
		= -\frac{2\pi i}{3}\kappa_{cd}f_{ab}^d(p_1 + p_2).
\end{align*}
Since this expression is already antisymmetric it is equal to $\Omega_{abc}$. This matches the expression for $\Omega$ obtained by a different derivation in \cite[\S10.2.3]{Costello:2019tri}.

\begin{remark}
This example shows that the main theorem of this paper can be considered a vast generalisation of the well-known result that the harmonic map equation for surfaces mapping to Lie groups has a Lax pair formulation \cite{Poh76}. It would be interesting to see whether this can be leveraged into an understanding of the moduli space of solutions to the classical equations of motion, as was done for harmonic maps to Lie groups in \cite{Uhl89}.
\end{remark}


\subsection{The 1-loop $\beta$-function}\label{ss:CP1-betafnc}

Finally, I will give a sketch of some evidence for Conjecture
\ref{conj:cost}. Recall that this conjecture of Costello claims
that the one-loop $\beta$-function may be identified with a vector
field induced by a flow on the moduli space of Riemann surfaces
equipped with a holomorphic 1-form $\omega$ of the type we have
been considering, together with a decomposition of the zeroes of
$\omega$ into two equally sized groups, $D_1$ and $D_2$.

In our situation, the constraints on the flow are that:
\begin{itemize}
	\item $\mc{P}_1 = \oint_{|z|=\epsilon}\omega$ is fixed, and
	\item $\mc{P}_2 = \int_{p_1}^{p_2}\omega$ varies to first order in the flow parameter $\varepsilon$ as $\mc{P}_2\to\mc{P}_2 + \varepsilon$.
\end{itemize}
We can achieve this by flowing the points $p_1, p_2$ linearly,
\begin{align}\label{eq:moduli_flow}
	p_1 \to p_1 + \varepsilon c
		\quad\text{and}\quad
		p_2 \to p_2 - \varepsilon c
\end{align}
and making a judicious choice of constant $c$.

It is straightforward to calculate that the periods are given by
\begin{align}
	\mc{P}_1 &= -2\pi i (p_1 + p_2), \label{eq:closed_period}\\
	\mc{P}_2 &= 2(p_2 - p_1) + \frac{\mc{P}_1}{2\pi i}\log\left(\frac{p_2}{p_1}\right) \mod\mc{P}_1. \label{eq:open_period}
\end{align}
Under the flow \eqref{eq:moduli_flow} we have that
	\[	\mc{P}_2 \to \mc{P}_2 + \varepsilon c \frac{(p_1 - p_2)^2}{p_1p_2}
								+O(\varepsilon^2),	\]
so our constraints are satisfied if we take $c = \frac{p_1p_2}{(p_1 - p_2)^2}.$ Similarly, under this flow the metric coefficients vary as
\begin{align*}
	g_{ab} \to g_{ab} + 4\pi i \varepsilon c \kappa_{ab}
		= \left(1 + \varepsilon\frac{2 p_1p_2}{(p_1 - p_2)^3}\right)g_{ab}.
\end{align*}
By the definition of the one-loop $\beta$-function, we would therefore like to show that
	\[	\beta_{ab} \stackrel{\text{?}}{=}\frac{2 p_1p_2}{(p_1 - p_2)^3}g_{ab}.	\]

According to \cite{CFMP85}, the one-loop $\beta$-function for the
metric is
\begin{align}\label{eq:metric_beta}
	\frac{1}{\alpha'}\beta_{\mu\nu}
		&= \Ric_{\mu\nu} - \frac{1}{4}H_{\mu\lambda\kappa}H_\nu^{\phantom{\nu}\lambda\kappa},
\end{align}
where $H$ is proportional to $\Omega$, $H = \tilde{c}\Omega$. I will pin down the specific value of $\tilde{c}$ shortly.

The Ricci curvature is invariant under constant scalings of the metric, and so $\Ric_{ab} = -\frac{1}{4}\kappa_{ab}$. We therefore have that
\begin{align*}
	-\frac{4}{\alpha'}\beta_{ab}
		&= \kappa_{ab} + \tilde{c}^2g^{ce}g^{df}\Omega_{acd}\Omega_{bef} \\
		&= \kappa_{ab} + \frac{1}{(2\pi i)^2 (p_1 - p_2)^2}\frac{(2\pi i)^2 (p_1 + p_2)^2 \tilde{c}^2}{9}\underbrace{\kappa^{ce}\kappa^{df}\kappa_{dg}f_{ac}^g \kappa_{fh}f_{be}^h}_{-\kappa_{ab}} \\
		&= \left(
			1 - \frac{\tilde{c}^2}{9}\frac{(p_1+p_2)^2}{(p_1 - p_2)^2}
		\right)\kappa_{ab}
\end{align*}
To determine $\tilde{c}$ we observe that $\displaystyle\lim_{\substack{p_1 \to 0 \\ p_2 \to\infty}}\frac{\omega}{p_1 + p_2}$ is the 1-form that gives rise to the WZW model \cite{Costello:2019tri}, and that the $\beta$-function of that theory vanishes. The result of rescaling the 1-form is a constant rescaling of the metric and 3-form. The Ricci tensor is invariant under constant rescalings, as is the expression $g^{ce}g^{df}\Omega_{acd}\Omega_{bef}$ (the rescaling of the 3-forms cancels with the rescaling of the inverse metrics), so the $\beta$-function is unchanged by rescalings and we may simply pass to the $p_1\to 0, p_2\to\infty$ limit
	\[	1 - \frac{\tilde{c}^2}{9}\frac{(p_1+p_2)^2}{(p_1 - p_2)^2}
		\to 1 - \frac{\tilde{c}^2}{9} = 0	\]
i.e.\ $\tilde{c} = 3$. We then have
\begin{align*}
	-\frac{4}{\alpha'}\beta_{ab}
		&= \left(
			1 - \frac{(p_1+p_2)^2}{(p_1 - p_2)^2}
		\right)\kappa_{ab}
		= \frac{(p_1 - p_2)^2 - (p_1 + p_2)^2}{(p_1 - p_2)^2}\kappa_{ab}
		= -\frac{4p_1 p_2}{(p_1-p_2)^2}\kappa_{ab}
\end{align*}
so that
\begin{align*}
	\beta_{ab} &= \alpha'\frac{p_1p_2}{(p_1-p_2)^2}\kappa_{ab}
				= \frac{\alpha'}{2\pi i}\frac{p_1 p_2}{(p_1 - p_2)^3}g_{ab}
\end{align*}
which gives the desired result upon setting $\alpha' = 4\pi i$.

\begin{remark}\label{rk:more-conj-evidence}
An alternative proof of the conjecture in this example -- in fact, a proof of the conjecture for an arbitrary number of double poles on the Riemann sphere\footnote{Many thanks to Benoit Vicedo for this reference and observation.} -- can be deduced from the results of \cite{DLSS21}, in which the 1-loop RG flow of the models introduced in \cite{DLMV19a,DLMV19b} is completely determined.
\end{remark}


%

%

%% file: szego.bbl
\begin{thebibliography}{DLMV19b}

\bibitem[Arn89]{Arn89}
V.~I. Arnol'd.
\newblock {\em Mathematical methods of classical mechanics}, volume~60 of {\em
  Graduate Texts in Mathematics}.
\newblock Springer-Verlag, New York, second edition, 1989.
\newblock Translated from the Russian by K. Vogtmann and A. Weinstein.

\bibitem[BBT03]{BBT03}
Olivier Babelon, Denis Bernard, and Michel Talon.
\newblock {\em Introduction to classical integrable systems}.
\newblock Cambridge Monographs on Mathematical Physics. Cambridge University
  Press, Cambridge, 2003.

\bibitem[Byk16]{Byk16}
Dmitri Bykov.
\newblock Complex structures and zero-curvature equations for $\sigma$-models.
\newblock {\em Physics Letters B}, 760:341--344, 2016.

\bibitem[BZB03]{BZB03}
David Ben-Zvi and Indranil Biswas.
\newblock Theta functions and {S}zeg\"{o} kernels.
\newblock {\em Int. Math. Res. Not.}, (24):1305--1340, 2003.

\bibitem[CFMP85]{CFMP85}
C.~G. Callan, D.~Friedan, E.~J. Martinec, and M.~J. Perry.
\newblock Strings in background fields.
\newblock {\em Nuclear Phys. B}, 262(4):593--609, 1985.

\bibitem[Cos13]{C13}
Kevin Costello.
\newblock Supersymmetric gauge theory and the yangian, 2013.

\bibitem[CY19]{Costello:2019tri}
Kevin Costello and Masahito Yamazaki.
\newblock {Gauge Theory And Integrability, III}.
\newblock 8 2019.

\bibitem[DLMV19a]{DLMV19b}
F.~Delduc, S.~Lacroix, M.~Magro, and B.~Vicedo.
\newblock Assembling integrable {$\sigma$}-models as affine {G}audin models.
\newblock {\em J. High Energy Phys.}, (6):017, 86, 2019.

\bibitem[DLMV19b]{DLMV19a}
F.~Delduc, S.~Lacroix, M.~Magro, and B.~Vicedo.
\newblock Integrable coupled $\ensuremath{\sigma}$ models.
\newblock {\em Phys. Rev. Lett.}, 122:041601, Jan 2019.

\bibitem[DLSS21]{DLSS21}
Fran\c{c}ois Delduc, Sylvain Lacroix, Konstantinos Sfetsos, and Konstantinos
  Siampos.
\newblock R{G} flows of integrable {$\sigma$}-models and the twist function.
\newblock {\em J. High Energy Phys.}, (2):Paper No. 065, 43, 2021.

\bibitem[FT07]{FT86}
Ludwig~D. Faddeev and Leon~A. Takhtajan.
\newblock {\em Hamiltonian methods in the theory of solitons}.
\newblock Classics in Mathematics. Springer, Berlin, english edition, 2007.
\newblock Translated from the 1986 Russian original by Alexey G. Reyman.

\bibitem[LV21]{LV21}
Sylvain Lacroix and Beno\^{\i}t Vicedo.
\newblock Integrable {$\mathcal{E}$}-models, 4d {C}hern-{S}imons theory and
  affine {G}audin models. {I}. {L}agrangian aspects.
\newblock {\em SIGMA Symmetry Integrability Geom. Methods Appl.}, 17:Paper No.
  058, 45, 2021.

\bibitem[{nLa}21]{nlab:grothendieck_connection}
{nLab authors}.
\newblock {{G}}rothendieck connection.
\newblock \url{http://ncatlab.org/nlab/show/Grothendieck%20connection}, April
  2021.

\bibitem[Poh76]{Poh76}
K.~Pohlmeyer.
\newblock Integrable {H}amiltonian systems and interactions through quadratic
  constraints.
\newblock {\em Comm. Math. Phys.}, 46(3):207--221, 1976.

\bibitem[Uhl89]{Uhl89}
Karen Uhlenbeck.
\newblock Harmonic maps into {L}ie groups: classical solutions of the chiral
  model.
\newblock {\em J. Differential Geom.}, 30(1):1--50, 1989.

\bibitem[Wri15]{Wr15}
Alex Wright.
\newblock Translation surfaces and their orbit closures: an introduction for a
  broad audience.
\newblock {\em EMS Surv. Math. Sci.}, 2(1):63--108, 2015.

\bibitem[Zar19]{Zar19}
K.~Zarembo.
\newblock Integrability in sigma-models.
\newblock In {\em Integrability: from statistical systems to gauge theory},
  pages 205--247. Oxford Univ. Press, Oxford, 2019.

\end{thebibliography}
